\documentclass[12pt, a4paper]{amsart}

\usepackage{amsmath}
\usepackage{amsfonts}
\usepackage{amssymb}
\usepackage{amsthm}
\usepackage{color, soul}
\usepackage{dsfont}
\usepackage[all]{xy}
\usepackage{enumerate}
\usepackage{verbatim}
\makeatletter
\@namedef{subjclassname@2020}{%
  \textup{2020} Mathematics Subject Classification}
\makeatother

\makeatletter
\@namedef{subjclassname@2020}{%
  \textup{2020} Mathematics Subject Classification}
\makeatother

\newtheorem{theorem}{Theorem}[section] % reset theorem numbering for each chapter
\newtheorem{definition}[theorem]{Definition}
\newtheorem{lemma}[theorem]{Lemma}
\newtheorem{prop}[theorem]{Proposition}
\newtheorem{corollary}[theorem]{Corollary}
\newtheorem{remark}[theorem]{Remark}
\newtheorem{example}[theorem]{Example}
\newtheorem*{definition_No_nu}{Definition}
\theoremstyle{definition}

\def\Ae{\mbox{\AE}}
\def\MFIN{\mbox{MFIN}}
\def\FIN{\mbox{FIN}}
\def\COFIN{\mbox{COFIN}}
\def\ep{\varepsilon}
\def\Lip{{\rm Lip}}
\def\ult{\mathcal U}
\def\A{\mathcal A}
\def\B{\mathcal B}

\def\F{\mathcal F}
\def\OF{\overline{\F}}
\def\I{\mathcal I}
\def\L{\mathcal L}
\def\J{\mathcal J}
\def\W{\mathcal W}

\def\min{{\rm min}}
\def\max{{\rm max}}
\def\reg{{\rm reg}}
\def\sur{{\rm sur}}
\def\spa{{\rm span}}
\def\ker{{\rm Ker}}

\begin{document}
\title[]{On the Lipschitz operator ideal $\Lip_{0}\circ \A\circ \Lip_{0}$}
\author{Nahuel Albarrac\'{\i}n, Pablo Turco}


\address{IMAS - UBA - CONICET - Pab I,
Facultad de Cs. Exactas y Naturales, Universidad de Buenos
Aires, (1428) Buenos Aires, Argentina}

\email{nalbarracin@dm.uba.ar}
\email{paturco@dm.uba.ar}

\begin{abstract}
We study a systematic way to produce a Lipschitz operator ideal from a Banach linear operator ideal $\A$. For maximal and minimal operator ideals $\A$, the Lipschitz maximal hull and minimal kernel of the Lipschitz operator ideals $\Lip_0 \circ \A \circ \Lip_0$ are investigated, respectively. 
Using ultraproduct techniques, we obtain the Lipschitz version of a result of K{\"u}rsten and Piestch which characterizes the maximal hull of any Lipschitz operator ideal. Among other results, we characterize the linear operators which belong to $\Lip_0\circ \A\circ \Lip_0$ which, in many cases, they are precisely those which are in $\A$. In particular, we give some cases in which a nonlinear factorization of linear operators implies a linear one, in terms of a given Banach linear operator ideal $\A$.
\end{abstract}
\maketitle

\section*{Introduction}
In the last years, several authors have studied different classes of Lipschitz maps between Banach spaces that have, in some way, analogous properties of well-known linear Banach operator ideals (in the sense of Pietsch). For instance, Farmer and Johnson \cite{FJ} introduced the Lipschitz $p$-summing operators and showed that they possess a Lipschitz version of the Pietsch factorization theorem. Also, the linear operators which are Lipschitz $p$-summing are exactly the (linear) $p$-summing operators. In the article of Johnson, Maurey and Schechtman \cite{JMS}, the class of Lipschitz mappings which factorize through an arbitrary $L_p$ space is introduced and they showed that if a linear operator from a Banach space to a dual space can be factorized through an arbitrary $L_p$ space via two Lipschitz maps, then it can be factorized through $L_p$ via two linear operators. This allowed them to show that if a Banach space is uniformly or coarsely equivalent to a $\L_1$ space, then it is linearly equivalent to a $\L_1$ space (i.e. is a $\L_1$ space) \cite[Corollary~4]{JMS}, answering a question posted by Heinrich and Mankiewicz.  Since then, many articles were dedicated to the study of different classes of Lipschitz maps defined between metric or Banach spaces which {\it extend} different classes of linear operators (see e.g. \cite{ADT, ARY, BeChe, CPJV, ChaDo12, ChaDo,  CZ, JVSMVV14, KhaSa}, among others). The necessity to study these different classes in a general framework, unifying results and the language, leads to the new concept of Banach Lipschitz operator ideal (see definition below). In general, these ideals {\it extend} the  Banach linear operator ideals. It is convenient to introduce what we mean by extension.

\begin{definition_No_nu} Let $\A$ be a class of linear operators and $\I$ a class of Lipschitz operators. We say that $\I$ extends $\A$ for the Banach spaces  $E$ and $F$ if  
$$
\I \cap \L(E,F)=\A(E,F).
$$
In other words, the linear operators from $E$ to $F$ which belong to $\I$ are exactly those which belong to $\A$.
\end{definition_No_nu}

As was pointed out by Farmer and Johnson respect $p$-summing operators, a  natural question arises \cite[Problem~6]{FJ}. Let $\A$ be a Banach (linear) operator ideal and $\I$ a Banach Lipschitz operator ideal that extends $\A$. What results about $\A$ have analogues for $\I$?. 

From a fixed Banach operator ideal $\A$, there is a canonical way to obtain a Banach Lipschitz operator ideal. Following \cite[Definition~3.1]{ARSPY}, for a (pointed) metric space $X$ and a Banach space $E$, a Lipschitz map $f\colon X\rightarrow E$ belongs to the composition ideal $\A\circ \Lip_0(X,E)$ if there exists a Banach space $G$, a Lipschitz function $g\colon X\rightarrow G$ and a linear operator $T\in \A(G,E)$ such that $f=T\circ g$. A  Banach Lipschitz operator ideal $\I$ is of {\it composition type} if there exists a Banach operator ideal $\A$ such that $\I=\A \circ \Lip_0$. This procedure to obtain a Banach Lipschitz operator ideal from $\A$, extends $\A$ for separable Banach spaces \cite[Proposition~3.2]{TV}. Also, in the case when $\A$ is a maximal Banach operator ideal, $\A\circ \Lip_0$ extends $\A$ for any Banach spaces $E$ and $F$ \cite[Proposition~3.3]{TV}. In other words, if $T$ is a linear operator which can be factorized as $T=R\circ f$ where $f$ is a Lipschitz map and $R$ is a linear operator which belongs to $\A$, then $T$ belongs to $\A$. However, in general there are different Banach Lipschitz operator ideals which extend the same Banach (linear) operator ideal. For example, there are at least 2 Banach Lipschitz operator ideals which extend the ideal of $p$-summing operators \cite[Proposition~3.17]{TV}.

The main objective of the article is to study some properties of Banach Lipschitz operator ideals obtained from Banach linear operator ideals as follows. Given a Banach operator ideal $\A$, a pointed metric space $X$ and a Banach space $E$, the Lipschitz ideal $\Lip_0\circ \A \circ \Lip_0(X,E)$ consists of all Lipschitz functions $f\colon X\rightarrow E$ such that there exist Banach spaces $G_1, G_2$, Lipschitz functions $g_1\colon X\rightarrow G_1$, $g_2\colon G_2\rightarrow E$ and a linear map $T\in \A(G_1,G_2)$ such that $f=g_2\circ T \circ g_1$. We will focus in the case when $\A$ is a minimal or maximal Banach operator ideal. We will see that in most of the cases, this type of Lipschitz ideals are not of composition type.

The article is divided as follows. In Section~\ref{Sec: Background} we fix the notation, state the definitions and general results that will be used. In Section~\ref{Sec: Minimal} we will consider the case when $\A$ is minimal. We will show that $\Lip_0 \circ \A \circ \Lip_0$ is not minimal as a Lipschitz operator ideal (Corollary~\ref{Coro: Nominimal}) and we will show in which cases $\Lip_0 \circ \A \circ \Lip_0$ and its minimal kernel extends $\A$ (Corollary~\ref{Coro:1raext} and Theorem~\ref{Theo: lipAlipmin}). In particular, we obtain that, under some assumption on the Banach spaces $E$ or $F$, if $T\colon E\rightarrow F$ is a linear operator which can be factorized as $T=f_2\circ R\circ f_1$ where $f_1$ and $f_2$ are Lipschitz maps and $R$ is a linear operator which belongs to $\A$, then $T$ belongs to $\A$. 

The main result in Section~\ref{Sec: Maximal} is a characterization of maximal Banach Lipschitz operator ideals extending the result of K{\"u}rsten and Pietsch for Banach operator ideals: A Banach Lipschitz operator ideal is maximal if and only if it is regular and ultrastable (Theorem~\ref{Theo:_Reg_and_Ultrast}). This result allows us to describe the maximal hull of $\Lip_0 \circ \A \circ \Lip_0$ when $\A$ is maximal.

 In Section~\ref{Sec: Restriction} we investigate which Banach operator ideal do $\Lip_0\circ \A \circ \Lip_0$ and its maximal hull extend when the ideal $\A$ is maximal. We show that for some maximal Banach operator ideals $\A$, $\Lip_0 \circ \A \circ \Lip_0$ extends $\A$. In other words, under some assumption over $\A$, if $T\colon E\rightarrow F$ is a linear operator which can be factorized as $T=g_2\circ R\circ g_1$ where $g_1$ and $g_2$ are Lipschitz maps and $R$ is a linear operator which belongs to $\A$, then $T$ belongs to $\A$. (Theorem~\ref{thm: Factorization}). This type of result implies that for a class of linear operators, certain non-linear factorization of them implies a linear one. This can be compared with results obtained in \cite{CZ, FJ,JMS}. We finish with a brief discussion on when $\Lip_0\circ \A\circ\Lip_0$ is of composition type or not.

We refer to Pietsch's book \cite{Pie} and the book of Defant and Floret \cite{DF} for the theory of Banach (linear) operator ideals, to the book of Weaver \cite{Wea} for the theory of Lipschitz maps and to the book of Benyamini and Lindenstrauss \cite{BL} for the theory of non-linear geometry of Banach spaces.

\section{Notation and general background}\label{Sec: Background}

Throughout the manuscript, $X, Y$ will be metric spaces and $d_X, d_Y$ their distance functions, respectively. $E$ and $F$ will denote real Banach spaces with the norms $\|{\cdot}\|_E$ and $\|{\cdot}\|_F$ respectively. Whenever the space is understood, we will simply write $d$ or $\|{\cdot}\|$. The open unit ball of $E$ will be denote by $B_E$, and the dual and bidual space of $E$ are denoted by $E'$ and $E''$. The canonical inclusion of $E$ into $E''$ is denoted by $J_E$. A pointed metric space is a metric space with a distinguished point, that we will always denote by $0$. In particular, a normed vector space is a pointed metric space and its distinguished point will be $0$. As usual, the space of all linear continuous, weakly compact, approximable and finite-rank  operators between $E$ and $F$ will be denoted by $\L(E,F), \W(E,F), \OF(E,F)$ and $\F(E,F)$ respectively. In the sequel, we will need the notion of regular Banach operator ideal. The regular hull of a Banach operator ideal $\A$, denoted as $\A^{\reg}$, are those operators $T\colon E\rightarrow F$ such that $J_F\circ T\in \A(E,F'')$, together with the norm $\|T\|_{\A^{\reg}}=\|J_F\circ T\|_\A$.  We say that $\A$ is regular if it coincides with its regular hull.

Recall that for two metric spaces $X$ and $Y$, a map $f\colon X\rightarrow Y$ is said to be a \textit{Lipschitz map} if there exists a constant $C>0$ such that $d(f(x_1),f(x_2))\leq C d(x_1,x_2)$ for all $x_1, x_2 \in X$. The least of such constants will be denoted as $\Lip(f)$. We denote by $\Lip_0(X,Y)$ the set of all Lipschitz maps from $X$ to $Y$ that vanish at $0$. If we consider a Banach space $E$, the space $\big(\Lip_0(X,E);\Lip(\cdot)\big)$ becomes a Banach space. In the case when $E=\mathds R$, we write $\Lip_0(X,\mathds R)=X^\#$ and it will be called as the \textit{Lipschitz dual} of $X$. We refer to the book of Weaver \cite{Wea} for more about this space.
For any pointed metric space $X$, the {\it Arens-Eells space} (or the Lipschitz free space over $X$) will be denoted as $\Ae(X)$.  If we consider the Dirac map $\delta_X\colon X \rightarrow (X^\#)'$ defined as $\delta_X(x)(f)=f(x)$, we have the equality
$$
\Ae(X)=\overline{\spa\big\{\delta_X(x) \colon x \in X\big\}}\subset (X^\#)'
$$
(see e.g. \cite[Definition~1.1]{GK}). By \cite{AE56}, we have that $X^\#$ is a dual Banach space and $\Ae(X)'=X^\#$. 

For $f\in\Lip_0(X,Y)$, there exists a unique linear operator $\widehat f\in\L(\Ae(X),\Ae(Y))$ such that $\delta_Y\circ f=\widehat f\circ \delta_X$ (see e.g. \cite[Theorem~3.6]{Wea}). When $E$ is a Banach space, there is a linear quotient map $\beta_E \colon \Ae(E)\rightarrow E$, called the \textit{barycenter map}, which is a left inverse of $\delta_E$, meaning $\beta_E \circ \delta_E=Id_E$. There is an isometric isomorphism between $\Lip_0(X,E)$ and $\L(\Ae(X),E)$. Indeed, given $f\in \Lip_0(X,E)$, there exists a unique linear operator $L_f \in \L(\Ae(X),E)$ such that $f=L_f\circ \delta_X$ with $\Lip(f)=\|L_f\|$ (see e.g. \cite[Proposition~2.2.4]{Wea}). The linear operator $L_f$ is defined as $L_f=\beta_E\circ\widehat f$ and satisfies $L_f(\delta_{X}(x))=f(x)$. We will refer the operator $L_f$ as \textit{the linearization} of $f$. 

A function between Banach spaces $f\colon E \rightarrow F$  is said to be Gateaux differentiable at $x_0 \in E$ if there exists an operator $T\in \L(E,F)$ such that 
$$
T(x)=\lim_{t\rightarrow 0} \dfrac{ f(x_0+tx)-f(x_0)}t
$$
for every $x\in E$. The operator $T$ is called the {\it Gateaux derivative of $f$ at $x_0$} and it is denoted by $Df(x_0)$. Note that, in particular if $f$ is Lipschitz and Gateaux differentiable at $x_0$ then $\|Df(x_0)\|\leq \Lip_0(f)$.

By a \textit{Banach Lipschitz operator ideal} we mean a subclass $\I_{\Lip}$ of $\Lip_0$ such that for every pointed metric space $X$ and every Banach space $E$, the components
$$
\I_{\Lip}(X,E)=\Lip_0(X,E)\cap \I_{\Lip}
$$
satisfy:
\begin{enumerate}[\upshape (i)]
\item $\I_{\Lip}(X,E)$ is a linear subspace of $\Lip_0(X,E)$.
\item $Id_\mathds{R} \in \I_{\Lip}(\mathds{R},\mathds{R})$.
\item The ideal property: if $g \in \Lip_0(Y,X), f\in \I_{\Lip}(X,E)$ and $S \in \L(E,F)$, then  $S\circ f\circ g \in \I_{\Lip}(Y,F)$.
\end{enumerate}
Also there is a {\it Lipschitz ideal norm} over $\I_{\Lip}$, given by a function $\|{\cdot}\|_{\I_{\Lip}}\colon \I_{\Lip}\rightarrow [0,+\infty)$ that satisfies

\begin{enumerate}[\upshape (i')]
\item For every pointed metric space $X$ and every Banach space $E$, the pair \newline $\big(\I_{\Lip}(X,E);\|{\cdot}\|_{\I_{\Lip}}\big)$ is a Banach space and $\Lip(f)\leq \|f\|_{\I_{\Lip}}$ for all $f \in \I_{\Lip}(X,E)$.
\item $\|{Id_{\mathds R} \colon \mathds R\rightarrow \mathds R}\|_{\I_{\Lip}}=1$
\item If $g \in \Lip_0(Y,X), f\in \I_{\Lip}(X,E)$ and $S \in \L(E,F)$, then \newline $\|S\circ  f\circ  g\|_{\I_{\Lip}} \leq \Lip(g) \|f\|_{\I_{\Lip}} \|S\|$.
\end{enumerate}

This definition was introduced in \cite[Definition~2.1]{ARSPY} and independently in \cite[Definition~2.3]{CPJVVV}, under the name of {\it generic Lipschitz operator Banach ideal} and extends the definition of Banach linear operator ideals (see e.g. \cite{Pie}).
Along the manuscript, with Banach operator ideal we will refer to a Banach {\it linear} operator ideal. To avoid confusion, we will denote $\A$, $\B$ for Banach operator ideals while $\I$, $\J$ stands for Banach Lipschitz operator ideals. With $\A\subset \B$  we mean that for all Banach spaces $E$ and $F$, $\A(E,F)\subset \B(E,F)$, and $\|\cdot\|_{\B}\leq \|\cdot\|_{\A}$. The same applies with Lipschitz operator ideals. In particular, for a Banach operator ideal $\A$, $\A \circ \Lip_0$ and $\Lip_0 \circ \A \circ \Lip_0$ are Banach Lipschitz operator ideals endowed with the norms
$$
 \|f\|_{\A\circ \Lip_0}=\inf\{\|T\|_{\A} \Lip(g)\colon f=T\circ g\}
$$ 
and 
$$\|f\|_{\Lip_0\circ \A\circ \Lip_0}=\inf\{\Lip (g_1)\|T\|_{\A} \Lip(g_2)\colon f=g_2\circ T\circ g_1\},$$ respectively.
 
Given two Banach Lipschitz operator ideals $\I$ and $\J$ (Banach operator ideals $\A$ and $\B$), we say that $\I$ and $\J$ (resp. $\A$ and $\B$) are {\it related} if for every finite pointed metric space $X_0$ and every finite-dimensional Banach space $E_0$, the equality $\I(X_0,E_0)=\J(X_0,E_0)$ holds isometrically (resp. for all finite-dimensional Banach spaces $E_0$ and $F_0$, the equality $\A(E_0,F_0)=\B(E_0,F_0)$ holds isometrically). As it happens in the linear case, in \cite{TV} it is shown that, for a Banach Lipschitz operator ideal $\I$, there exists the {\it biggest} and the {\it smallest} Banach Lipschitz operator ideal related with $\I$, which are called the maximal hull and the minimal kernel of $\I$. This is, there exist unique $\I^{\max}$ and $\I^{\min}$ such that are related with $\I$ and, whenever $\J$ is related with $\I$, then $\I^{\min}\subset \J \subset \I^{\max}$.

All other relevant terminology and preliminaries are given in corresponding sections.

\section{About $\Lip_{0}\circ \A^{\min}\circ \Lip_{0}$}\label{Sec: Minimal}

Our firsts results appeal to a result of Johnson, Maurey and Schechtman \cite[Theorem~1]{JMS}. For our purpose, we decide state it with a slight modification.
\begin{theorem}[Johnson--Maurey--Schechtman]
\label{thm: JMS lemma about differentiability}
Let $E$ be a Banach space, $F$ be a dual Banach space, and $T\in \L(E,F)$. If there exist a Banach space $G$, a Lipschitz function $f\in \Lip_0(E,G)$ Gateaux differentiable at some $x_0 \in E$ and $g\in \Lip_0(G,F)$ such that $T=g\circ f$, then there exists $R\in \L(G,F)$ with $\|R\|\leq \Lip(g)$ such that $T = R\circ Df(x_0)$.
\end{theorem}

The following corollary will be useful later.
\begin{corollary}\label{Coro: JMS} 
Let $E$ and $F$ be Banach spaces and $T\in \L(E,F)$. If there exist a Banach space $G$, a Lipschitz function $f\in \Lip_0(E,G)$ Gateaux differentiable at some $x_0 \in E$ and $g\in \W \circ \Lip_0(G,F)$ such that $T=g\circ f$, then there exists an operator $R\in \L(G,F)$ with $\|R\|\leq \Lip(g)$ such that $T = R \circ Df(x_0)$.
\end{corollary}

\begin{proof}
Take the factorization of $T\in \L(E,F)$ as in the statement. Since $g\in \W \circ \Lip_0(G,F)$, applying \cite[Proposition~2.2]{JVSMVV14} there exist a Banach space $G_1$, a Lipschitz operator $g_1\in \Lip_0(G,G_1)$ and an operator $S\in \W(G_1,F)$ such that $g=S\circ g_1$ with $\|S\|\Lip(g_{1})= \Lip(g)$. Using the isometric version of the Davis, Figiel, Johnson and Pelcz\'ynski factorization lemma due to Lima, Nygaard and Oja \cite{LNO}, we may suppose that $G_1$ is reflexive. Consider the canonical factorization  $S=\overline S \circ q$ where $q\colon G_1 \rightarrow G_1/ \ker S$ is the quotient map. Note that $G_1/ \ker S$ is a reflexive Banach space and $\overline S\in \L(G_1/ \ker S,F)$ is injective. We have the factorization $T=\overline S \circ q \circ g_1 \circ f$. Since $T$ is linear and $\overline S$ is linear and injective, then $q\circ g_1\circ f\colon E\rightarrow G_1/ \ker  S$ is a linear operator. As  $G_1/ \ker S$ is reflexive, by the above theorem, there exists a linear operator $\widetilde R\in \L(G, G_1/ \ker S)$ with $\|\widetilde R\|\leq \Lip (q\circ g_1)\leq \Lip(g_1)$ such that $q\circ g_1\circ f=\widetilde R\circ Df(x_0)$. Finally, we obtain that $T=\overline S \circ \widetilde R \circ Df(x_0)$ and 
$$
\|\overline S\circ \widetilde R \|\leq \|\overline S\|\|\widetilde R\|\leq\|S\|\Lip(g_1)\leq\Lip(g). 
$$
The proof concludes by taking $R=\overline S\circ \widetilde R$.
\end{proof}

Recall that given a Banach operator ideal $\A$, its minimal kernel can be obtained via the composition ideal $\A^{\min}=\OF\circ \A \circ\OF$. Since Johnson \cite{Joh} proved that every approximable linear operator factors through a separable and reflexive Banach space via two approximable linear operators, we may consider a factorization of an operator in $\A^{\min}$ through reflexive and separable Banach spaces. Now we have
\begin{prop}\label{prop: a(min)lip0 fact}

Let $E$ and $F$ be Banach spaces and $\A$ a Banach operator ideal. If $E$ is separable, then for every Lipschitz map $f\in \A^{\min} \circ \Lip_0 (E,F)$, there exists $x_0$ in $E$ such that $f$ is Gateaux differentiable at $x_0$ and $Df(x_0) \in \A^{\min}(E,F)$. Moreover, for $\ep>0$ we may choose $x_0$ such that $\|Df(x_0)\|_{\A^{\min}}\leq (1+\ep)\|f\|_{\A^{\min}\circ \Lip_0}$.
\end{prop}
\begin{proof}
Fix $\ep>0$ and take a factorization $f=T\circ g$ through some separable and reflexive Banach space $G$, where $g\in \Lip_0(E,G)$ with $\Lip(g)=1$ and $T\in \A^{\min}(G,F)$ with $\|T\|_{\A^{\min}}\leq (1+\ep) \|f\|_{\A^{\min}\circ \Lip_0}$. Then, since $g$ is a Lipschitz function from a separable space into a reflexive space (hence with the Radon-Nikod\'ym property), by \cite[Theorem~6.42]{BL} $g$ is Gateaux differentiable at some $x_0$. Then, $f$ is Gateaux differentiable in $x_0$ with  $Df(x_0)=T\circ Dg(x_0)$, and the conclusion follows.
\end{proof}

Combining Theorem~\ref{thm: JMS lemma about differentiability} with Proposition~\ref{prop: a(min)lip0 fact} we have
\begin{theorem}
\label{Thm: linear restriction of lipaminlip is contained in amin reg}
Let $E$ and $F$ be Banach spaces and $\A$ a Banach operator ideal. If $E$ is separable, then $(\Lip_0\circ \A^{\min}\circ \Lip_0)\cap \L(E,F)\subset (\A^{\min})^{\emph{reg}}(E,F)$.
\end{theorem}
\begin{proof}
Take a linear operator $T\in \Lip_0\circ \A^{\min}\circ \Lip_0(E,F)$ and fix $\ep>0$. We may take a factorization of $T=f\circ g$, where $f\in \Lip_0(G,F)$ and $g\in \A^{\min}\circ \Lip_0(E,G)$ and,  thanks to Proposition~\ref{prop: a(min)lip0 fact}, we may consider it Gateaux differentiable at some $x_0\in E$ with $Dg(x_0) \in \A^{\min}(E,G)$. Moreover, we may choose such factorization with $\Lip(f)=1$ and $\|Dg(x_0)\|_{\A^{\min}} \leq (1+\ep)\|T\|_{\Lip_0\circ \A^{\min}\circ \Lip_0}$. Now, applying Theorem~\ref{thm: JMS lemma about differentiability} to the operator $J_F\circ T=J_F\circ f\circ g$, we obtain that $J_F\circ T\in \A^{\min}(E,F'')$ and $\|J_F\circ T\|_{\A^{\min}}\leq (1+\ep)\|Dg(x_0)\|_{\A^{\min}} \leq (1+\ep)^2\|T\|_{\Lip_0\circ \A^{\min}\circ \Lip_0}$, and the conclusion follows.
\end{proof}
In the cases when $\A^{\min}=\A^{\min \ \reg}$ we obtain an extension of a minimal ideal.

\begin{corollary}\label{Coro:1raext}
Let $\mathcal{A}$ be a Banach operator ideal, $E$ and $F$ be Banach spaces, with $E$ separable. If $F$ is a dual space, $E'$ has the approximation property, or $\mathcal{A}^{\min}$ is regular, then
\begin{align*}
(\Lip_{0}\circ \A^{\min}\circ \Lip_{0})(E,F)\cap \L(E,F)=\A^{\min}(E,F)\text{ isometrically}
\end{align*}
\end{corollary}

\begin{proof}
Following \cite[Corollary 22.8.2]{DF}, if $F$ is a dual space or $E'$ has the approximation property, then $\A^{\min}(E,F)=\A^{\min \ \reg}(E,F)$ and the proof is complete.
\end{proof}

The following result gives another extension of $\A^{\min}$.

\begin{theorem}\label{Theo: WlipAlipmin}
Let $E$ and $F$ be Banach spaces and $\A$ a Banach operator ideal. If $E$ is separable, then $(\W\circ \Lip_{0}\circ \A^{\min}\circ \Lip_{0})\cap \L(E,F)=  \A^{\min}(E,F)$.
\end{theorem}
\begin{proof}

First note that by \cite[Proposition~3.1]{TV} we have that $\A^{\min}\subset (\A^{\min}\circ \Lip_0)\cap \L$. Also, we have that $\A^{\min}=\OF \circ \A^{\min}\subset \W\circ \A^{\min}$. Thus, $\A^{\min}\circ\Lip_0\subset \left(\W\circ \Lip_{0}\circ \A^{\min}\circ \Lip_{0}\right)$.

Now, take $T\in (\W\circ \Lip_{0}\circ \A^{\min}\circ \Lip_{0})\cap \L(E,F)$ and let see that $T\in \A^{\min}(E,F)$. For $\ep>0$, there exist a Banach space $G$, and Lipschitz maps $f\in \A^{\min}\circ \Lip_0(E,G)$ and $g\in \W\circ \Lip_0(G,F)$ such that $T=g\circ f$ and $\|f\|_{\A^{\min}\circ \Lip_0} \Lip(g) \leq (1+\ep)\|T\|_{\W\circ \Lip_{0}\circ \A^{\min}\circ \Lip_{0}}$. Combining Corollary~\ref{Coro: JMS} with Proposition~\ref{prop: a(min)lip0 fact}, there exist $x_0 \in E$ such that $f$ is Gateaux differentiable in $x_0$, $Df(x_0)\in \A^{\min}(E,G)$ with $\|Df(x_0)\|_{\A^{\min}}\leq (1+\ep)\|f\|_{\A^{\min} \circ \Lip_0}$ and $R\in \L(E,G)$ with $\|R\|\leq \Lip(g)$ such that $T=R \circ Df(x_0)$. We conclude that $T\in \A^{\min}(E,F)$ and $\|T\|_{\A^{\min}}\leq (1+\ep)\|f\|_{\A^{\min} \circ \Lip_0}\Lip(g)\leq (1+\ep)^2\|T\|_{\W\circ \Lip_{0}\circ \A^{\min}\circ \Lip_{0}}$.
\end{proof}

The results obtained can be rewritten in terms of factorizations of linear operators.

\begin{theorem}\label{thm: Factorization_Min}
Let $E$ and $F$ be  Banach spaces, $\A$ a minimal Banach operator ideal and $T\in \L(E,F)$. Suppose that there is a factorization of $T$ as 
$$
\xymatrix{
E\ar[r]^T \ar[d]_{f_1}& F\\
G_1 \ar[r]^R & G_2 \ar[u]_{f_2}
}
$$
where $G_1$ and $G_2$ are Banach spaces, $f_1\in \Lip_{0}(E,G_{1})$, $f_{2}\in \Lip_{0}(G_{2},F)$ and $R\in \A(G_1,G_2)$.
If $E$ is separable and
\begin{itemize}
\item $f_2 \in \W\circ \Lip_0(G_2,F)$,
\end{itemize}
or
\begin{itemize}
\item $F$ is a dual space, $E'$ has the approximation property, or $\A$ is regular,
\end{itemize}
then $T \in \A(E,F)$.
\end{theorem}

To obtain another extension of $\A^{\min}$, we will use the following factorization result about the minimal kernel of a Banach Lipschitz operator ideal.

\begin{prop}\label{Prop: Sobre Imin}
Let $\I$ be a Banach Lipschitz operator ideal. Then $\I^{\min}=\OF \circ \I^{\min}$.
 
This is, for every pointed metric space $X$, every Banach space $E$ and every $f\in \I^{\min}(X,E)$ there exist a Banach space $G$, a Lipschitz function $g\in \I^{\min}(X,G)$ and a linear operator $T\in \OF(G,E)$ such that $f=T\circ g$. Moreover, $\|f\|_{\I^{\min}}=\inf\{\|g\|_{\I^{\min}} \|T\|\}$, where the infimum is taken over all the factorizations of $f$. 
\end{prop}
\begin{proof}
It is clear that $\OF\circ \I^{\min}\subset \I^{\min}$. To prove the equality, it suffices to show that $\OF\circ \I^{\min}(X_{0},E_{0})=\I^{\min}(X_{0},E_{0})$ isometrically for every finite pointed metric space $X_{0}$ and every finite-dimensional normed space $E_{0}$. Given $f\in\Lip_0(X_{0},E_{0})$, since $E_0$ is finite-dimensional, the identity map of $E_0$ is a finite-rank operator. Then, we can write $f=id_{E_{0}}\circ f$, which gives a factorization of $f$ with $\|f\|_{\OF\circ \I^{\min}}\leq \|id_{E_{0}}\|\|f\|_{\I^{\min}}=\|f\|_{\I^{\min}}$. The proof is finished.
\end{proof}

\begin{remark} For a Banach operator ideal $\A$, we have that $\A^{\min}=\OF \circ \A \circ \OF$, (see e.g. \cite[8.6.1]{Pie}). As far as we know, there is no formula for a minimal Lipschitz ideal analogous to this one. Note that, as a consequence of the above, we have the inclusion
\begin{equation}\label{eq1}
(\Lip_0\circ \A^{\min}\circ \Lip_0)^{\min}=\OF \circ (\Lip_0\circ \A^{\min}\circ \Lip_0)^{\min}\subset \OF \circ (\Lip_0\circ \A^{\min}\circ \Lip_0).
\end{equation}
Since $\A^{\min}=\OF \circ \A\circ \OF$, from $\eqref{eq1}$ we obtain that
$$
(\Lip_0\circ \A^{\min}\circ \Lip_0)^{\min}\subset (\OF \circ \Lip_{0}\circ \OF) \circ \A\circ (\OF\circ\Lip_{0})\subset (\OF \circ \Lip_{0}) \circ  \A\circ (\OF\circ\Lip_{0}),
$$ 
where $\OF \circ \Lip_{0}$ is the Lipschitz ideal of approximable operators. In fact, $\OF \circ \Lip_{0}$ is the {\it smallest} Lipschitz ideal which contains the approximable linear operators. 
\end{remark}

Another extension of $\A^{\min}$ is described in the following theorem.

\begin{theorem}\label{Theo: lipAlipmin}
Let $E$ and $F$ be Banach spaces and $\A$ a Banach operator ideal. If $E$ is separable, then $(\Lip_{0}\circ \A^{\min}\circ \Lip_{0})^{\min}\cap \L(E,F)=  \A^{\min}(E,F)$.
\end{theorem}
\begin{proof}
On the one hand, since $\A^{\min}\circ \Lip_0 \subset \Lip_0\circ \A^{\min}\circ \Lip_0$ and by \cite[Theorem~4.8]{TV} $\A^{\min}\circ \Lip_0$ is a minimal Lipschitz operator ideal, we obtain the inclusion
\begin{align*}
\A^{\min}\circ \Lip_{0}=(\A^{\min}\circ \Lip_{0})^{{\min}}\subset (\Lip_{0}\circ \A^{\min}\circ \Lip_{0})^{{\min}}.
\end{align*}
On the other hand, note that by Proposition \ref{Prop: Sobre Imin},
$$
(\Lip_{0}\circ \A^{\min}\circ \Lip_{0})^{\min}\subset \OF\circ \Lip_{0}\circ \A^{\min}\circ \Lip_{0}\subset \W\circ \Lip_{0}\circ \A^{\min}\circ \Lip_{0}.
$$
Thus $\A^{\min}\circ \Lip_{0}\subset \W\circ \Lip_{0}\circ \A^{\min}\circ \Lip_{0}$, and the result follows by applying \cite[Proposition~3.1]{TV} and Theorem~\ref{Theo: WlipAlipmin}.
\end{proof}

Combining Theorem~\ref{Theo: lipAlipmin} and \cite[Proposition~3.2]{TV}, when $E$ is a separable Banach space, for every Banach space $F$, $(\Lip_0\circ \A^{\min}\circ \Lip_0)^{\min}(E,F)$ and $\A^{\min}\circ \Lip_0(E,F)$ extend $\A^{\min}(E,F)$ and, for example, in the case that $F$ is a dual space, by Corollary~\ref{Coro:1raext}, also $\Lip_0\circ \A^{\min}\circ \Lip_0(E,F)$ extends it. To finish this section, we give a discussion about in which cases these three ideals coincide or not. For a first approach, recall that a Lipschitz function belongs to the composition ideal $\A \circ \Lip_0$ if and only if its linearization belongs to $\A$, \cite[Proposition~3.2]{ARSPY}. Also, in \cite[Definition~1.6]{CPJVVV} the concept of {\it strong} Banach Lipschitz operator ideal is introduced. A Banach Lipschitz operator ideal $\I$ has the {\it strong ideal property} if whenever $X$ and $Y$ are pointed metric spaces, $E$ and $F$ Banach spaces, $f \in \I(Y,E)$, $g\in \Lip_0(X,Y)$ and $h\in \Lip_0(E,F)$, we have that $h\circ f\circ g \in \I(X,F)$ with $\|h\circ f\circ g\|_{\I} \leq \Lip(h) \|f\|_\I \Lip(g)$. In particular, $\Lip_0 \circ \A \circ \Lip_0$ is a Banach Lipschitz operator ideal with the strong ideal property.

\begin{prop}\label{prop:strong prope}
Let $\I$ be a Banach Lipschitz operator ideal. If $\I$ has the strong ideal property, then $\I$ is not minimal.
\end{prop}
\begin{proof}
Let $E\neq \{0\}$ be a finite-dimensional normed space. Since $id_{E}\in \I(E,E)$ (because it is a finite-rank operator) and $\I$ has the strong ideal property, we have that $\delta_{E}=\delta_{E}\circ id_{E}\in \I(E, \Ae(E))$. If $\I$ were minimal, as a consequence of \cite[Proposition~4.5]{TV}, we would obtain that $\I\subset  \OF\circ \Lip_{0}$, which would imply that the linearization of $\delta_{E}$, which is the identity map of $\Ae(E)$, is an approximable (and hence compact) operator. But this cannot happen since $\Ae(E)$ is an infinite-dimensional space. 
\end{proof}
As a consequence, we have the following
\begin{corollary}\label{Coro: Nominimal}
Let $\A$ be a Banach operator ideal. The Lipschitz operator ideal $\Lip_{0}\circ \A^{\min}\circ \Lip_{0}$ is not minimal.
\end{corollary}
In particular, $\Lip_{0}\circ \A^{\min}\circ \Lip_{0}$ does not coincide with $(\Lip_{0}\circ \A^{\min}\circ \Lip_{0})^{\min}$ or $\A^{\min} \circ \Lip_0$, since $\A^{\min} \circ \Lip_0$ is a minimal Lipschitz operator ideal \cite[Theorem~4.8]{TV}. Almost with the same proof of Proposition~\ref{prop:strong prope} we obtain the following.

\begin{prop}\label{Prop: No composition}
Let $\A$ be a Banach operator ideal. The Lipschitz operator ideal $\Lip_{0}\circ \A^{\min}\circ \Lip_{0}$ is not of composition type
\end{prop}
\begin{proof}
If $\Lip_{0}\circ \A^{\min}\circ \Lip_{0}$ were of composition type, then by \cite[Proposition 3.16]{TV}, we would have
\begin{align*}
((\Lip_{0}\circ \A^{\min}\circ \Lip_{0})\cap \L)\circ \Lip_{0}=\Lip_{0}\circ \A^{\min}\circ \Lip_{0}.
\end{align*}
In particular, for $E\neq \{0\}$ a finite-dimensional Banach space, since $\delta_E \in \Lip_{0}\circ \A^{\min}\circ \Lip_{0}(E,\Ae(E))$, we get that its linearization, $id_{\Ae(E)}$, belongs to $(\Lip_{0}\circ \A^{\min}\circ \Lip_{0})\cap \L (\Ae(E),\Ae(E))$. By Theorem~\ref{Thm: linear restriction of lipaminlip is contained in amin reg}, we have that $id_{\Ae(E)}\in (\A^{\min})^{\reg}(\Ae(E),\Ae(E))$ and, in particular, $id_{\Ae(E)}$ is a compact operator. But this cannot happen since $\Ae(E)$ is an infinite-dimensional space. 
\end{proof}

Finally, in order to see in which cases $(\Lip_0\circ \A^{\min}\circ \Lip_0)^{\min}$ and $\A^{\min}\circ \Lip_0$ coincide or not, we have the following
\begin{prop}
Let $\A$ be a Banach operator ideal. The following statements are equivalent:
\begin{itemize}
\item[$(i)$] $(\Lip_{0}\circ \A^{\min}\circ \Lip_{0})^{\min}$ is of composition type.
\item[$(ii)$] $(\Lip_{0}\circ \A^{\min}\circ \Lip_{0})^{\min}=\A^{\min}\circ \Lip_{0}$.
\item[$(iii)$] $(\Lip_{0}\circ \A^{\min}\circ \Lip_{0})^{\max}=\A^{\max}\circ \Lip_{0}$.
\end{itemize}
\end{prop}
\begin{proof}
It is clear that $(ii)$ implies $(i)$ and the equivalence between $(ii)$ and $(iii)$ follows from \cite[Proposition  4.1, Proposition 4.4 and Theorem 4.8]{TV}. We will show that $(i)$ implies $(ii)$. First note that, since both ideals are minimal, is enough to show the equality $(\Lip_{0}\circ \A^{\min}\circ \Lip_{0})^{\min}(X_0,E_0)=\A^{\min}\circ \Lip_{0}(X_0,E_0)$ for every finite metric space $X_0$ and every finite-dimensional space $E_0$. Suppose that  $(\Lip_{0}\circ \A^{\min}\circ \Lip_{0})^{\min}$ is of composition type, then by \cite[Proposition 3.16]{TV}, we have
$$
\left((\Lip_{0}\circ \A^{\min}\circ \Lip_{0})^{\min}\cap \L\right) \circ \Lip_{0}=(\Lip_{0}\circ \A^{\min}\circ \Lip_{0})^{\min}.
$$
By Proposition~\ref{Theo: lipAlipmin} we have that for all finite metric space $X_0$ and every finite-dimensional normed space $E_0$, $(\Lip_{0}\circ \A^{\min}\circ \Lip_{0})^{\min}\cap \L (\Ae(X_0),E_0)=\A^{\min}(\Ae(X_0),E_0)$, implying that 
$(\Lip_{0}\circ \A^{\min}\circ \Lip_{0})^{\min}(X_0,E_0)=\A^{\min}\circ \Lip_{0}(X_0,E_0).
$
\end{proof}

\begin{prop}\label{Prop: Minimal of composition type}
Let $\A$ be a Banach operator ideal. If $(\Lip_{0}\circ \A^{\min}\circ \Lip_{0})^{\min}$ is of composition type, then $id_{\Ae(E)}\in \A^{\max}(\Ae(E), \Ae(E))$ for every finite-dimensional normed space $E$.
\end{prop}
\begin{proof}

As it was done in Proposition~\ref{Prop: No composition}, if $E$ is a finite-dimensional space, we have that $\delta_E \in \Lip_{0}\circ \A^{\min}\circ \Lip_{0}(E,\Ae(E))$. In particular, $\delta_{E}$ belongs to $(\Lip_{0}\circ \A^{\min}\circ \Lip_{0})^{\max}(E,\Ae(E))$, and by the above proposition we obtain that $\delta_E \in \A^{\max}\circ \Lip_{0}(E,\Ae(E))$. This implies that the linearization of $\delta_E$, which is $id_{\Ae (E)}$, belongs to $\A^{\max}(\Ae(E),\Ae(E))$.
\end{proof}

Consequently, if $\A$ is a Banach operator ideal such that the identity of $\Ae(E)$ does not belong to $\A^{\max}(\Ae(E),\Ae(E))$ for some finite-dimensional normed space $E$, then $(\Lip_{0}\circ \A^{\min}\circ \Lip_{0})^{\min}$ and $\A^{\min}\circ \Lip_{0}$ do not coincide. A special case in which the coincidence happens is when $\A$ is a closed Banach operator ideal.

\begin{prop}\label{prop: minimal closed} Let $\A$ be a closed Banach operator ideal. Then $(\Lip_0 \circ \A \circ \Lip_0)^{\min}$ is of composition type. Moreover, $(\Lip_0 \circ \A \circ \Lip_0)^{\min}=\OF \circ \Lip_0$.
\end{prop}
\begin{proof}
Since $(\Lip_{0}\circ \A\circ \Lip_{0})^{\min}$ is minimal,  by \cite[Proposition 4.5]{TV} we have that $(\Lip_{0}\circ \A\circ \Lip_{0})^{\min}\subset \OF\circ \Lip_{0}$. On the other hand, we have $\A\circ \Lip_{0}\subset \Lip_{0}\circ \A\circ \Lip_{0}$, which implies that
$$
\A^{\min}\circ \Lip_{0}\subset (\Lip_{0}\circ \A\circ \Lip_{0})^{\min}.
$$
Since $\A$ is closed, we know that $\A^{\min}=\OF$, and the proof is complete.
\end{proof}
In particular, for the ideal of approximable operators,  by Proposition~\ref{Prop: No composition},  $\Lip_0 \circ \OF \circ \Lip_0$ is not of composition type, but by Proposition~\ref{prop: minimal closed} its minimal kernel is of composition type, giving a negative answer of \cite[Problem~4.15]{TV}.

We do not know any other example besides $\OF$ of a minimal Banach operator ideal such that $(\Lip_0 \circ \A^{\min} \circ \Lip_0)^{\min}$ is of composition type. However, we have the following result. Before stating it, recall that, following \cite[Definition~4.7.1]{Pie}, an operator $T\colon E\rightarrow F$ belongs to the {\it surjective hull} of a Banach operator ideal $\A$, denoted as $\A^{\sur}$, if $T\circ Q_E \in \A(\ell_1(B_E),F)$ where $Q_E\colon \ell_1(B_E)\rightarrow E$ is the canonical quotient map.

\begin{prop}\label{prop:No_coincide} Let $\A$ be a Banach operator ideal. If $(\Lip_0\circ \A^{\min}\circ \Lip_0)^{\min}$ if of composition type, then $(\A^{\max})^{\sur}= \L$.
\end{prop}
\begin{proof}
Fix a Banach operator ideal $\A$ and suppose that $(\Lip_0\circ \A^{\min}\circ \Lip_0)^{\min}$ is of composition type. By Proposition~\ref{Prop: Minimal of composition type}  we have that $id_{\Ae(\mathbb R)} \in \A^{\max}(\Ae(\mathbb R), \Ae(\mathbb R))$. Since $\Ae(\mathbb R)$ is isometric to $L_1(\mathbb R)$ (which is a projectively universal separable Banach space since it contains a complemented copy of $\ell_1$), for every separable Banach space $E$ there exists a bounded linear operator from $\Ae(\mathbb R)$ onto $E$. Call such operator $T\colon \Ae(\mathbb R)\rightarrow E$, which belongs to $\A^{\max}(\Ae(\mathbb R),E)$ and by the usual diagram 
$$
\xymatrix{
\Ae(\mathbb R) \ar@{->>}[d]_q\ar[r]^T& E\\
\Ae(\mathbb R)/\ker(T) \ar[ur]^{\overline T} &
}
$$
we have that $\overline T$ is a isomorphism which belongs to $\A^{\max \ \sur}(\Ae(\mathbb R)/\ker(T),E)$ (see \cite[9.8]{DF}). By the ideal property, we deduce that  $id_{E}\in \A^{\max \ \sur}(E,E)$. Therefore, the ideal of separable operators (those which have separable image) is included in $\A^{\max \ \sur}$. By \cite[Proposition 4.2.5]{Pie}, the ideal of separable operators is closed. Then, taking maximal hulls and using that $\A^{\max \ \sur}=\A^{\sur \ \max}$ \cite[Proposition~8.7.14]{Pie}, we get that $\A^{\max \ \sur}=\L$, and the proof is completed.
\end{proof}

In particular, if $\A$ is the maximal Banach ideal of $p$-integral, $p$-summing, $p$-dominated for $1\leq p<\infty$ or the maximal Banach ideal of $p$-factorable operators with $1<p\leq \infty$, by the above $(\Lip_0\circ \A^{\min}\circ \Lip_0)^{\min}$ is not of composition type. The next results shows that for the maximal ideal of $1$-factorable operators $\mathfrak{L}_1$, $(\Lip_0 \circ \mathfrak{L}_1^{\min} \circ \Lip_0)^{\min}$ is not of composition type, in spite of the fact that $\mathfrak{L}_1^{\sur}=\L$, implying that the converse of Proposition~\ref{prop:No_coincide} does not hold.

\begin{prop} The ideal $(\Lip_0 \circ \mathfrak{L}_1^{\min} \circ \Lip_0)^{\min}$ is not of composition type.
\end{prop}
\begin{proof}
Suppose that $(\Lip_0 \circ \mathfrak{L}_1^{\min} \circ \Lip_0)^{\min}$ is of composition type. Then, using Proposition \ref{Prop: Minimal of composition type}, we have that the identity operator of $\Ae(\mathbb{R}^{2})$ is $1$-factorable. By \cite[Corollary~18.6.1]{DF}, the canonical inclusion of $\Ae(\mathbb{R}^{2})$ into its bidual can be factorized via $L_1(\mu)$ for some measure $\mu$. From this it follows that $\Ae(\mathbb{R}^{2})$ is isomorphic to a subspace of $L_{1}(\mu)$. This contradicts a result of Naor and Schechtman (see also \cite[Remark 4.2]{CDW}). Consequently, $(\Lip_{0}\circ \mathfrak{L}_{1}^{\min}\circ \Lip_{0})^{\min}$ is not of composition type, as we wanted to show.
\end{proof}

\section{Maximality and ultrastability of Lipschitz operator ideals}\label{Sec: Maximal}

In what follows, for a pointed metric space $X$, $\MFIN(X)$ denotes the family of all finite metric subsets of $X$ with ``0'', meanwhile for a Banach space $E$, $\FIN(E)$ and $\COFIN(E)$ denotes the family of all finite- and cofinite-dimensional subspaces of $E$ respectively. For a subset $X_0$ of $X$ we denote by $\iota_{X_0}^{X}$ the canonical inclusion of $X_0$ into $X$ and, for a subspace $L$ of $E$, $q_L^E\colon E\rightarrow E/L$ is the canonical quotient map.
 The theory of ultraproducts in Banach spaces was introduced by Dacunha--Castelle and Krivine \cite{DaCKr}. For the basics of this theory we refer to \cite{Hei}. Let $(X_i)_{i \in I}$ be a family of pointed metric spaces with base point $0_i \in X_i$ and let $(E_i)_{i \in I}$ be a family of Banach spaces. Take an ultrafilter $\ult$ in $I$. We denote by $(E_i)_{\ult}$ the ultraproduct of the family of Banach spaces $(E_{i})_{i\in I}$ with respect to $\ult$. Similarly, we denote by $(X_{i})_{\ult}$ the so-called ultralimit of the $X_{i}$'s with respect to $\ult$ (see \cite[Definition 7.19]{Roe}), which is a pointed metric space with base point $(0_i)_{\ult}$. When all the $X_{i}$'s are Banach spaces, the ultralimit coincides with the Banach space ultraproduct. The norm of an element $(x_i)_{\ult}$ in $(E_i)_{\ult}$ can be computed as $\|(x_i)_{\ult}\|=\lim_{\ult}\|x_i\|$, meanwhile the distance of $(x_i)_{\ult}$ to $(y_i)_{\ult}$ in $(X_i)_{\ult}$ is $d((x_i)_{\ult},(y_i)_{\ult})=\lim_{\ult} d(x_i, y_i)$. For a family of Lipschitz mappings $f_i\colon X_i\rightarrow E_i$ with $\sup_{i \in I}\Lip(f_i)\leq C$ for some $C>0$, the rule
$$
(f_i)^{\ult}(x_i)_{\ult}\colon =(f_ix_i)_{\ult}
$$
defines a Lipschitz map between $(X_i)_{\ult}$ and $(E_i)_{\ult}$ with Lipschitz norm less or equal to $C$. In the case that the family $(X_i)_{i\in I}$ are Banach spaces and $(f_i)_{i \in I}$ consists of linear operators, $(f_i)^{\ult}$ is linear as well.

Recall that a Banach operator ideal $\A$ is {\it ultrastable} if for all families of Banach spaces $(E_i)_{i \in I}, (F_i)_{i \in I}$ and linear operators $S_i\in \A(E_i,F_i)$  with $\|S_i\|_{\A}\leq C$ for some $C>0$ for all $i\in I$, we have that for any ultrafilter $\ult$ of $I$, the linear operator $(S_i)^{\ult}$ belongs to $\A((E_i)_{\ult}, (F_i)_{\ult})$ and $\|(S_i)^{\ult}\|_{\A}\leq C$. We now extend the notion of ultrastable to Lipschitz ideals in verbatim.

\begin{definition} A Banach Lipschitz operator ideal $\I$ is said to be ultrastable if for any families of pointed metric spaces $(X_{i})_{i\in I}$, Banach spaces $(E_{i})_{i\in I}$,  every family  of Lipschitz operators $f_{i}\in \I(X_{i},E_{i})$ such that $\sup\limits_{i\in I}\|f_{i}\|_{\mathcal{I}}<\infty$ and every ultrafilter $\ult$ in $I$, the Lipschitz operator $(f_i)^{\ult}$ belongs to $\I((X_{i})_{\ult}, (E_{i})_{\ult})$ and $\|(f_i)^{\ult}\|_{\I} \leq \lim_{\ult} \|f_i\|_{\I}$.
\end{definition}

The following two results show us a way to construct ultrastable Lipschitz operator ideals.

\begin{prop}\label{prop:ultrastable a derecha} Let $\A$ be a Banach operator ideal. If $\A$ is ultrastable, then $\A\circ \Lip_0$ is ultrastable.
\end{prop}
\begin{proof}
Take a family of pointed metric spaces $(X_i)_{i \in I}$, of Banach spaces $(E_i)_{i \in I}$ and of Lipschitz operators $(f_i)_{i\in I}$ such that $f_i\in \A\circ \Lip_{0}(X_i, E_i)$ for every $i \in I$ and $\sup\limits_{i\in I}\|f_{i}\|_{\A\circ \Lip_{0}}<\infty$. Fix $\varepsilon >0$ and for each $i \in I$, take a factorization  of $f_i=T_{i}\circ g_{i}$, where $g_{i}\in \Lip_{0}(X_{i},F_{i})$ with $\Lip(g_{i})\leq 1$, $T_{i}\in \A(F_{i},E_{i})$ with $\|T_{i}\|_{\A}=(1+\varepsilon)\|f_{i}\|_{\A\circ \Lip_{0}}$, for some Banach spaces $F_i$. Note that $
(f_{i})^{\ult}=(T_{i})^{\ult}\circ (g_{i})^{\ult}$. Since $\A$ is ultrastable, we get that $(T_{i})^{\ult}\in \A$, implying that $(f_{i})^{\ult}\in\A\circ \Lip_{0}((X_{i})_{\ult},(E_{i})_{\ult})$, with
\begin{align*}
\|(f_{i})^{\ult}\|_{\A\circ \Lip_{0}}\leq \|(T_{i})^{\ult}\|_{\A} \Lip((f_{i})^{\ult})\leq \lim\limits_{\ult}\|T_{i}\|_{\A}=(1+\varepsilon) \lim\limits_{\ult}\|f_{i}\|_{\A\circ \Lip_{0}}
\end{align*}
and the conclusion follows.
\end{proof}

The following proposition shows a way to construct an ultrastable Banach Lipschitz operator ideal with the strong ideal property from an ultrastable Banach Lipschitz operator ideal $\I$. For a pointed metric space $X$ and a Banach space $E$, we denote by $\Lip_0 \circ \I(X,E)$ the space of Lipschitz operators $f\colon X\rightarrow E$ for which there exist a Banach space $F$, and Lipschitz operators $g_1\in \I(X,F)$,  $g_2 \in \Lip_0(F,E)$ such that $f=g_2\circ g_1$. If we endow this space with the norm $\|f\|_{\Lip_0 \circ \I}=\inf \{\|g_1\|_{\I} \Lip_0(g_2)\}$, where the infimum is taken over all such factorizations of $f$, it becomes a Banach Lipschitz operator ideal.

\begin{prop}\label{prop:ultrastable a izquierda}Let $\I$ be a Banach Lipschitz operator ideal. If $\I$ is ultrastable, then $\Lip_0\circ \I$ is ultrastable.
\end{prop}
 % Mas aun, si  componemos 2 ideales ultraestables, la composicion es ultraestable
\begin{proof}

Take a family pointed metric spaces $(X_i)_{i \in I}$, a family of Banach spaces $(E_i)_{i\in I}$ and for each $i\in I$ consider Lipschitz operators $f_{i}\in \Lip_0(X_i,E_i)$ such that $\sup\limits_{i\in I}\|f_{i}\|_{\Lip_{0}\circ \I}<\infty$. Fix $\varepsilon>0$ and for each $i\in I$, write $f_{i}=h_{i}\circ g_{i}$, where $\Lip(h_{i})\leq 1$, $g_{i}$ is a Lipschitz operator  which belongs to $\I$ with $\|f_{i}\|_{\I}\leq \|g_{i}\|_{\Lip_{0}\circ \I}+\varepsilon$. Since
\begin{align*}
(f_{i})^{\ult}=(h_{i})^{\ult}\circ (g_{i})^{\ult}
\end{align*}
and $\I$ is ultrastable, we get that $(g_{i})^{\ult}\in \I$, implying that $(f_{i})^{\ult}\in \Lip_{0}\circ \I$, with
\begin{align*}
\|(f_{i})^{\ult}\|_{\Lip_{0}\circ \I}\leq \Lip((h_{i})^{\ult})\|(g_{i})^{\ult}\|_{\I}\leq \lim\limits_{\ult}\|g_{i}\|_{\I}\leq \varepsilon+\lim\limits_{\ult}\|f_{i}\|_{\Lip_{0}\circ \I}.
\end{align*}
and the conclusion follows.
\end{proof}

As a direct consequence of the two previous results, if $\A$ is an ultrastable Banach operator ideal, then $\Lip_0\circ \A \circ \Lip_0$ is an ultrastable Banach Lipschitz operator ideal.

Now, we turn our attention to maximal Lipschitz operator ideal. Analogous to what happens in the linear case, the maximal hull of a Banach Lipschitz operator ideal is defined by its properties when is considered over finite metric spaces. Following \cite{CPJVVV}, a Lipschitz operator $f\in \Lip_0(X,E)$ belongs to the maximal hull of a Banach Lipschitz operator ideal $\I$ if 
$$
\|f\|_{\I^{\max}}\colon =\sup\{\|q_L^E \circ f\circ  \iota_{X_0}^X\|_\I \colon X_0 \in \MFIN (X), L\in \COFIN (E)\} < \infty.$$

Also, $\|\cdot\|_{\I^{\max}}$ is the norm of $\I^{\max}$. As it was shown by K{\"u}rsten~\cite{Kur} and Pietsch~\cite{Pie2} for maximal Banach operator ideals, its turn out that a Banach Lipschitz operator ideal is maximal if and only if is ultrastable and regular. In order to see this, we will need some previous results.

First, we will show that a Lipschitz operator can be reconstructed from its {finite-dimensional parts, in an analogous way that is done for linear operators in \cite[Lemma~8.8.4]{Pie}. For this, take a Lipschitz operator $f \in \Lip_0(X,E)$ and consider the cartesian product $I:=\MFIN (X)\times \COFIN (E)$ endowed with the directed partial order defined as $(X_{1},L_{1})\leq (X_{2}, L_{2})$ iff $X_{1}\subseteq X_{2}$ and $L_{2}\subseteq L_{1}$. Choose an ultrafilter $\ult$ in $I$ containing the order filter. For each index $i\in I$, we denote by $X_{i}$ and $L_{i}$ the components of the pair $i$. For each $i\in I$, write $f_i=q_{L_i}^E\circ f\circ \iota_{X_i}^X$. We have $\Lip(f_{i})\leq \Lip(f)$ for all $i\in I$, and therefore the map $(f_{i})^{\ult}$ is defined. With this notation, we have

\begin{lemma}\label{lemma: reconstructing a lip operator from its finite parts} There are operators $g\in \Lip_0(X, (X_i)_{\ult})$ and $Q\in \L((E/L_i)_{\ult},E'')$ such that $\Lip(g)\leq 1$, $\|Q\|\leq 1$ and $J_E \circ f=Q \circ (f_i)^{\ult}\circ g$.
\end{lemma}
\begin{proof}
Let $g\colon X\rightarrow (X_{i})_{\ult}$  be the map defined as $g(x)=(x_{i})_{\ult}$, where
\begin{align*}
x_{i}=\begin{cases}
x & \text{ if }x_{i}\in X_{i}\\
0 & \text{ if }x_{i}\notin X_{i}
\end{cases}
\end{align*}
The map $g$ is an isometry. Indeed, given $x, y\in X$, we have
\begin{align*}
d(g(x), g(y))=d((x_{i})_{\ult}, (y_{i})_{\ult})=\lim\limits_{\ult}d(x_{i},y_{i}).
\end{align*}
If we fix $i_{0}\in I$ such that $x, y \in X_{i_{0}}$ (for example, $i_{0}=\{0, x, y\}\times E$), then $d(x_{i},y_{i})=d(x,y)$ for all $i\geq i_{0}$ and consequently $d(g(x), g(y))=\lim\limits_{\ult}d(x_{i},y_{i})=d(x,y)$.

We define the bounded linear operator $Q\colon (E/L_{i})_{\ult}\rightarrow E^{''}$ as
\begin{align*}
 Q((z_{i})_{\ult}) (e')=\lim\limits_{\ult} e'(y_{i})
\end{align*}
where $e'\in E'$, for some $y_{i}\in E$ such that $q_{L_i}^{E}(y_i)=z_i$ and, $\|y_i\|\leq C$ for some $C>0$ for all $i\in I$ (this choice is possible because the $z_{i}$'s are uniformly bounded). As it was done in \cite[Lemma~8.8.4]{Pie}, the operator is well-defined and $\|Q\|\leq 1$. Now, let $x\in X$ be given. We have
\begin{align*}
Q(f_{i})^{\ult}g(x)=Q(f_{i}(x_{i}))_{\ult}=Q((q_{L_i}^{E}(f(x_{i})))_{\ult}).
\end{align*}
For $e'\in E'$, we have
\begin{align*}
 Q((q_{L_i}^{E}(f(x_{i})))_{\ult})(e')=\lim\limits_{\ult}e'(f(x_{i})).
\end{align*}
Finally, since for all $i>i_0$, $f(x_i)=f(x)$, we have that $\lim\limits_{\ult}e'(f(x_{i}))=\lim\limits_{\ult}e'(f(x))$, and the conclusion follows.
\end{proof}

The following result should be compared with \cite[Proposition~6.1]{Hei} and \cite[Proposition~1.5]{FloHun}.
\begin{lemma}
\label{lemma: Lipschitz version of prop 6.1 heinrich}
Let $(X_{i})_{i\in I}$ be a family of pointed metric spaces and $\ult$ be an ultrafilter in $I$. For $\{0\}\neq Y\in \MFIN((X_{i})_{\ult})$,  there are, for all $i\in I$, Lipschitz maps $f_i\in \Lip_0(Y,X_i)$ such that
\begin{enumerate}[\upshape a)]
\item For all $x\in Y$, $x=(f_i(x))_{i \in \ult}$.
\item There exists $C>0$ such that $\Lip(f_i)\leq C$ for all $i\in I$.
\item  For $\varepsilon>0$, there exists $I_{0}\in \ult$ such that for each $i\in I_{0}$, $f_i\colon Y\rightarrow f_i(Y)$ are $(1+\varepsilon)$-Lipschitz isomorphisms.
\end{enumerate}
\end{lemma}

\begin{proof}
Suppose that $Y=\{0, x^{1},\dots,x^{n}\}$ and choose representations $x^{k}=(x_{i}^{k})_{\ult}$ ($1\leq k\leq n$) and for the base point $0=(0_{i})_{\ult}$. Denote $Z_{i}=\{0_i,x_i^1,\ldots,x_i^n\}\in \MFIN(X_{i})$. For each $i\in I$, let $f_{i}\colon Y\rightarrow Z_{i}$ be the base point preserving map defined as $f_{i}(x^{k})=x_{i}^{k}$ for $k=1,\dots,n$. 
Note that for all $1\leq k,k'\leq n$,  $\sup\limits_{i\in I}d(x_{i}^{k},x_{i}^{k'})$ is finite. Denote by $C\colon=\sup\limits_{\substack{1\leq k,k'\leq n\\ k\neq k'}} \frac{\sup\limits_{i\in I}d(x_{i}^{k},x_{i}^{k'})}{d(x^{k},x^{k'})}$. Then, we have
\begin{align*}
d(f_{i}(x^{k}),f_{i}(x^{k'}))=d(x_{i}^{k},x_{i}^{k'})\leq \sup\limits_{i\in I}d(x_{i}^{k},x_{i}^{k'})\leq C \, d(x^{k},x^{k'}).
\end{align*}
Therefore, $\Lip(f_{i})\leq C$ for all $i\in I$ and $x^k=(f_i(x^k))_{\ult}$. This gives $a)$ and $b)$.

Now, we can consider the Lipschitz map $(f_{i})^{\ult}\colon (Y)_{\ult}\rightarrow (X_{i})_{\ult}$. For all $1\leq k,k'\leq n$, we have that 
$$
\begin{array}{rl}
d((f_i)^{\ult}(x^{k})_{\ult},(f_i)^{\ult}(x^{k'})_{\ult})=&d((f_{i}x^{k})_{\ult},(f_{i}x^{k'})_{\ult})\\
=&\lim\limits_{\ult}d(f_{i}(x^{k}), f_{i}(x^{k'}))\\
=&\lim\limits_{\ult}d(x_{i}^{k},x_{i}^{k'})\\
=&d(x^{k},x^{k'}).
\end{array}
$$
This implies that, fixing $k$ and $k'$, the set
$$
I_{k,k'}:=\{i\in I: (1+{\varepsilon})^{-1}d(x^{k},x^{k'})\leq d(f_{i}x^{k},f_{i}x^{k'})\leq (1+\varepsilon)d(x^{k},x^{k'})\}
$$
belongs to the ultrafilter $\ult$. Since any filter is closed under finite intersections, the set $I_0=\bigcap\limits_{1\leq k,k'\leq n} I_{k,k'}$ also belongs to $\ult$. For each $i \in I_0$, we have
$$
(1+\varepsilon)^{-1}d(x^{k},x^{k'})\leq d(f_{i}x^{k},f_{i}x^{k'})\leq (1+\varepsilon)d(x^{k},x^{k'})
$$
for all $1\leq k,k'\leq n$. Denoting by $Y_{i}\colon=\{0,f_i(x^1),\ldots,f_i(x^n)\}\in \MFIN(X_{i})$, we obtain that the $f_i\colon Y\rightarrow Y_i$ are Lipschitz isomorphisms with $\Lip(f_{i})\leq 1+\varepsilon$ and $\Lip(f_{i}^{-1})\leq 1+\varepsilon$. This completes the proof.
\end{proof}

Finally, we need the description of maximal Lipschitz operator ideals in terms of Lipschitz tensor norms. We now recall some concepts and results. For the general background and theory of Lispchitz tensor product we refer to  \cite{CPCDJVVV2}. For $X$ a pointed metric space and $E$ a Banach space,  the Lipschitz tensor product $X \boxtimes E$ is the linear span of all linears functionals $\delta_{(x,y)} \boxtimes e$ on $\Lip_0(X, E')$ of the form $(\delta_{(x,y)} \boxtimes e)(f)=(f(x)-f(y))(e)$, for $x,y \in X$ and $e\in E$. As usual, we denote by $X\boxtimes_{\alpha} E$ the Lipschitz tensor product endowed with a norm $\alpha$ and by $X\widehat{\boxtimes}_{\alpha} E$ its completion.  A Lipschitz tensor norm $\alpha$ is a norm on the class of Lipschitz tensor products $X\boxtimes E$ which satisfies the following properties:
\begin{enumerate}[\upshape i)]
\item For all $x, y\in X$ and $e \in E$, $\alpha(\delta_{(x,y)} \boxtimes e)=d(x,y)\|e\|$.
\item For all Lipschitz function $g \in X^{\#}$ and $e' \in E'$, the functional $g\boxtimes e'$ on $X\boxtimes E$ defined as $g\boxtimes e'(\delta_{(x,y)} \boxtimes e)=(g(x)-g(y)) e'(e)$ belongs to $(X\boxtimes_{\alpha} E)'$ and $\|g\boxtimes e'\|\leq \Lip(g) \|e'\|$.
\item For every $f\in \Lip_0(X,Y)$ and $T\in \L(E,F)$, the linear map $f\boxtimes T\colon X\boxtimes_{\alpha}E\rightarrow Y\boxtimes_{\alpha}F$ defined as $f\boxtimes T(\sum_{j=1}^{n} \delta_{x_j,y_j}\boxtimes e_j)=\sum_{j=1}^{n} \delta_{f(x_j),f(y_j)}\boxtimes T(e_j)$ is continuous and $\|f\boxtimes T\|\leq \Lip(f)\|T\|$.
\end{enumerate}

For $u\in X\boxtimes E$, the $\alpha$-norm of $u$ in $X\boxtimes E$ is denoted by $\alpha(u;X,E)$. A Lipschitz tensor norm is said to be finitely generated if for all $u \in X\boxtimes E$, we have 
$$
\alpha(u;X,E)=\inf\{\alpha(u;X_0,E_0)\colon X_0\in \MFIN(X), E_0\in \FIN(E), u \in X_0\boxtimes E_0\}.
$$

Following the Representation Theorem for Maximal Banach Lipschitz operator ideals \cite[Corollary~5.2]{CPCDJVVV}, a Banach Lipschitz operator ideal $\I$ is maximal if and only if there exists a finitely generated Lipschitz tensor norm $\alpha$ such that for every pointed metric space $X$ and Banach space $E$, the equality 
$$
\I(X,E)=(X  \widehat{\boxtimes}_{\alpha}  E')'\cap \Lip_0(X,E)
$$
hold isometrically, where for $f\in \I(X,E)$, its associated linear functional $\phi_f$ is given by
$$
\phi_f(u)=\sum_{j=1}^{n} e_j'(f(x_j)-f(y_j)),\text{ for }u=\sum_{j=1}^{n} \delta_{x_j,y_j} \boxtimes e'_j \in X\boxtimes E'
$$

Now, we are ready to prove the main result of this section. Our proof follows the steps of \cite[Theorem~3.2]{FloHun}, where it is proved an analogous result for homogeneous polynomials.

\begin{theorem}\label{Theo:_Reg_and_Ultrast} A Banach Lipschitz operator ideal $\I$ is maximal if and only if it is regular and ultrastable. 
\end{theorem}

\begin{proof}
We use the notation from Lemma~\ref{lemma: reconstructing a lip operator from its finite parts}. Assume that $\I$ is regular and ultrastable and take a pointed metric space $X$, a Banach space $E$ and $f\in \I^{\max}(X,E)$. Then $\|f\|_{\I^{\max}}= \sup\limits_{i\in I}\|f_{i}\|_{\I}<\infty$. Since $\I$ is ultrastable, $(f_{i})^{\ult}\in \I((X_{i})_{\ult}, (E/L_{i})_{\ult})$ and by Lemma \ref{lemma: reconstructing a lip operator from its finite parts}, we have $J_{E}\circ f\in \I(X,E'')$. Then $f\in \I^{\reg}(X,E)=\I(X,E)$, with
\begin{align*}
\|f\|_{\I}=\|J_{E}\circ f\|_{\I}\leq \|(f_{i})^{\ult}\|_{\I}\leq \lim\limits_{\ult}\|f_{i}\|_{\I}\leq \sup\limits_{i\in I}\|f_{i}\|_{\I}=\|f\|_{\I^{\max}}.
\end{align*}

Conversely, assume that $\I$ is maximal. By \cite[Corollary 5.3]{CPJVVV}, we know that it is regular, therefore we only have to see that $\I$ is ultrastable. By the Representation Theorem for Maximal Banach Lipschitz operator ideals, there exists a finitely generated Lipschitz tensor norm $\alpha$ such that $\I(X,E)=(X\widehat \boxtimes_{\alpha}E')'\cap \Lip_0(X,E)$. Let $(X_i)_{i \in I}$ be a family of pointed metric spaces, $(E_i)_{i \in I}$ a family of Banach spaces and for each $i\in I$ take a Lipschitz function $f_i\in \I(X_i,E_i)$ such that $\sup\limits_{i\in I}\|f_{i}\|_{\I}<\infty$. Take an ultrafilter $\ult$ of $I$ and denote $X=(X_{i})_{\ult}$, $E=(E_{i})_{\ult}$ and $f=(f_{i})^{\ult}$. Since $\alpha$ is finitely generated, in order to see that $f\in \I(X,E)$ with $\|f\|_{\I}\leq \lim\limits_{\ult}\|f_{i}\|_{\I}$, it is enough to see that for every $Y\in \MFIN(X)$ and $M\in \FIN(E')$ and $u\in Y\boxtimes M$, we have
$$
|\phi_f(u)|\leq \lim\limits_{\ult}\|f_{i}\|_{\ult} \alpha(u;Y,M).
$$

Let $Y\in \MFIN(X)$ and $M\in \FIN(E')$ and $u\in Y\boxtimes M$. There exist $x^1,\ldots x^n, y^1,\ldots,y^n \in Y$ and $e'_1\ldots,e'_n\in M$ such that $u=\sum_{j=1}^{n} \delta_{x^j,y^j}\boxtimes e'_j$. Take $\ep>0$ and consider the set $\{f(x^1),\ldots, f(x^n), f(y^1),\ldots f(y^n)\}\in~\MFIN(E)$. Using the local duality of ultraproducts \cite[Theorem 7.3]{Hei}, there exists $T\in \L(M, (E'_i)_{i\in \ult})$ which is $(1+\ep)$-isomorphism onto its image such that $T(e'_k) (f(x^k)-f(y^k))=e'_k(f(x^k)-f(y^k))$ for $k=1,\ldots, n$. Then we have
\begin{equation}\label{eq12}
|\phi_f(u)|=\sum_{j=1}^n e'_j(f(x^j)-f(y^j))=\sum_{j=1}^n T(e'_j)(f(x^j)-f(y^j)).
\end{equation}

 Since $T(M)\in \FIN((E'_i)_{i\in \ult})$ by \cite[Proposition~6.1]{Hei} (see also \cite[Proposition~1.5]{FloHun}), for all $i \in \I$ there exists $S_i\in \L(T(M),E'_i)$ such that $T(e'_j)=(S_i(T(e'_j)))_{\ult}$ for $j=1,\ldots,n$ with $\|S_i\|\leq 1$. Denote $N_{i}=S_{i}(T(M))\in \FIN(E_{i}')$. We have
\begin{equation}\label{eq2}
\sum_{j=1}^n T(e'_j)(f(x^j)-f(y^j))=\sum_{j=1}^n (S_i (T(e'_j)))_{\ult}(f(x^j)-f(y^j))
\end{equation}

By Lemma~\ref{lemma: Lipschitz version of prop 6.1 heinrich}, there exists $g_i\in \Lip_0(Y,X_i)$ such that $x^j=(g_i(x^j))_{\ult}$ and  $y^{j}=(g_i(^j))_{\ult}$ for $j=1,\ldots,n$, $\Lip(g_i)<\widetilde C$ for some $\widetilde C>0$ and $I_0\in \ult$ such that for all $i\in I_0$, $g_i$ are $(1+\ep)$-Lipschitz isomorphisms onto its image. Denote by $Y_i=g_i(Y)\in \MFIN(X_i)$.  With this, we get

\begin{equation}\label{eq3}
\begin{array}{rl}
\displaystyle  \sum_{j=1}^n (S_i (T(e'_j)))_{\ult}(f(x^j)-f(y^j)) =& \displaystyle   \sum_{j=1}^n (S_i (T(e'_j)))_{\ult}(f((g_i(x^j))_{\ult})-f((g_i(y^j))_{\ult}))\\
=&\displaystyle  \sum_{j=1}^n (S_i (T(e'_j)))_{\ult}(f_i(g_i(x^j)))_{\ult})-(f_i(g_i(y^j)))_{\ult}).\\
=&\displaystyle \lim_{\ult} \sum_{j=1}^{n} S_i(T(e'_j)) (f_i(g_i(x^j))-f_i(g_i(y^j))).
\end{array}
\end{equation}

On the one hand, by the Representation Theorem of Maximal Banach Lipschitz operator ideals, taking into account that $N_i\in \FIN(E_i)$, we have $\I(Y_i,E_{i}/^\perp N_i)=(Y_i\boxtimes_{\alpha} N_i)'$, where $^\perp N_i\in \COFIN(E_{i})$ is the pre-annihilator of $N_i$. On the other hand, for each $i\in I$, let $u_i=\sum_{j=1}^{n} \delta_{g_i(x^j),g_i(y^j)} \boxtimes S_i(T(e'_j)) \in Y_i\boxtimes N_i$. Considering $\phi_{f_i}$ as a functional in $ (Y_i\boxtimes_{\alpha} N_i)'$, we have
\begin{equation}\label{eq4}
\begin{array}{rl}
\left|\sum_{j=1}^{n} S_i(T(e'_j)) (f_i(g_i(x^j))-f_i(g_i(y^j)))\right|=&|\phi_{f_i}(u_i)|\leq\|\phi_{f_i}\| \alpha(u_i;Y_i,N_i)\\
=&\|q^{E_i}_{^\perp N_i}\circ f_i\circ \iota_{Y_i}^{X_i}\|_{\I(Y_i,E_i/^\perp N_i)} \alpha(u_i;Y_i,N_i)\\
\leq & \|f_i\|_\I  \alpha(u_i;Y_i,N_i).
\end{array}
\end{equation}
Using the map $g_i\boxtimes (S_i\circ T)\colon Y\boxtimes E\rightarrow Y_i\boxtimes N_i$, we have that $g_i\boxtimes (S_i\circ T)(u)=u_i$, which implies that 
$$
\alpha(u_i;Y_i,N_i)\leq \Lip(g_i) \|S_i\circ T\|\alpha(u;X,E)\leq (1+\ep) \widetilde C \alpha(u;X,E)
$$
and therefore the limit $L=\lim\limits_{\ult}\alpha(u_i;Y_i,N_i)$ exists. Combining \eqref{eq12}, \eqref{eq2} ,\eqref{eq3} and \eqref{eq4} we obtain
$$
|\phi_f(u)|\leq \lim_{\ult} \|f_i\|_\I  \alpha(u_i;Y_i,N_i)=\lim_{\ult} \|f_i\|_\I \lim_{\ult} \alpha(u_i;Y_i,N_i). 
$$

By definition of convergence along an ultrafilter, there exists $I_1\in \ult$ such that for all $i\in I_1$, $|L-\alpha(u_i;Y_i,N_i)|\leq \ep$. Thus, taking $I_2=I_0\cap I_1 \in \ult$, we finally obtain that for all $i\in I_2, \ L\leq \left((1+\ep)^2 \alpha(u;X,E)+\ep\right)$ and then
$$
|\phi_f(u)|\leq \lim_{\ult} \|f_i\|_\I L\leq \lim_{\ult} \|f_i\|_\I \left((1+\ep)^2 \alpha(u;X,E)+\ep\right).
$$
Since $\ep>0$ was arbitrary, we obtain our result.
\end{proof}
\begin{remark}\label{remark:ultra=reg} Note that the first part of the proof shows that $\I^{\reg}=\I^{\max}$ for any ultrastable Banach Lipschitz operator ideal $\I$. This result in the case of Banach linear operators can be found in \cite[Proposition~8.8.6]{Pie}
\end{remark}

It is clear that if $\A$ is regular, then $\A \circ \Lip_0$ also is. Then, Proposition~\ref{prop:ultrastable a derecha}, together with Theorem~\ref{Theo:_Reg_and_Ultrast}, gives an alternative proof of \cite[Proposition~4.1]{TV}, where was shown that if $\A$ is maximal Banach operator ideal, then $\A \circ \Lip_0$ is a maximal Banach Lipschitz operator ideal.

To finish this section, we apply our results in order to characterize maximal hulls of some Banach Lipschitz opertor ideals. 

\begin{corollary}\label{coro: lipo amax lipo reg is maximal} 
If $\I$ is an ultrastable Banach Lipschitz operator ideal, then $(\Lip_0 \circ \I)^{\max}=(\Lip_0 \circ \I)^{\reg}$. In particular, $(\Lip_0 \circ \A^{\max} \circ \Lip_0)^{\max}=(\Lip_0 \circ \A^{\max}\circ \Lip_0)^{\reg}$ for any Banach operator ideal $\A$.
\end{corollary}
\begin{proof}
First note that by Proposition~\ref{prop:ultrastable a derecha}, $\A^{\max}\circ \Lip_0$ is ultrastable. Now, if $\I$ is ultrastable, by Proposition~\ref{prop:ultrastable a izquierda} we have that $(\Lip_0 \circ \I)$ is ultrastable, and the result follows from Remark~\ref{remark:ultra=reg}. 
\end{proof}

Farmer and Johnson \cite{FJ} introduced the ideal of Banach Lipschitz $p$-integral operators (see also \cite[p.5275]{CZ}). We are in condition to see that it is a maximal Banach Lipschitz operator ideal.
\begin{corollary}
The Banach Lipschitz ideal of Lipschitz $p$-integral operators $\I_{p}^{\Lip}$ is maximal.
\end{corollary}
\begin{proof}
Fix a pointed metric space $X$, a Banach space $E$ and a Lipschitz operator $f\colon X\rightarrow E$. Then $f$ is a Lipschitz $p$-integral operator if and only if $J_E\circ f \in \Lip_0\circ \I_p \circ \Lip_0(X, E'')$, where $\I_p$ is the maximal operator ideal of $p$-integral operators. This implies that $\I_{p}^{\Lip}=(\Lip_{0}\circ \I_{p}\circ \Lip_{0})^{\reg}$ and the result follows from Corollary \ref{coro: lipo amax lipo reg is maximal}.
\end{proof}

\section{About $\Lip_{0}\circ \A^{\max}\circ \Lip_{0}$}\label{Sec: Restriction}

The main goal of this section is to study which operator ideal do $\Lip_0\circ \A \circ \Lip_0$ and its maximal hull extend, in the case when $\A$ is a maximal Banach operator ideal. We start with some preliminary results which have their own interest.

\begin{lemma} \label{lemma: linear restriction of a maximal lipschitz ideal is maximal}
Let $\I$ be a maximal Banach Lipschitz operator ideal. Then $\I\cap \L$ is a maximal Banach operator ideal.
\end{lemma}

\begin{proof}
Let $E, F$ be Banach spaces and $T\in (\I\cap \L)^{\max}(E,F)$. By definition, we have
\begin{align*}
\|T\|_{(\I\cap \L)^{\max}}=\sup \{\|q_{F_{0}}^{F}\circ T\circ \iota_{E_{0}}^{E}\|_{\I\cap \L}: E_{0}\in \text{FIN}(E), F_{0}\in \text{COFIN}(F)\}
\end{align*}
where $\iota_{E_{0}}^{E}:E_{0}\hookrightarrow E$ is the inclusion map and $q_{F_{0}}^{F}:F\longrightarrow F/F_{0}$ is the quotient map. For each finite subset $X_{0}$ of $E$ containing the origin, the inclusion map $\iota_{X_{0}}^{E}:X_{0}\hookrightarrow E$ can be factorized as $\iota_{\langle X_{0}\rangle}^{E}\circ \iota_{X_{0}}^{\langle X_{0}\rangle}$, where $\langle X_{0}\rangle$ is the linear span of $X_{0}$. So, for each $X_{0}\in \text{MFIN}(E)$ and $F_{0}\in\text{COFIN}(F)$, we have
\begin{align*}
\|q_{F_{0}}^{F}\circ T\circ \iota_{X_{0}}^{E}\|_{\I}=\|q_{F_{0}}^{F}\circ T\circ \iota_{\langle X_{0}\rangle}^{E}\circ \iota_{X_{0}}^{\langle X_{0}\rangle}\|_{\I}\leq \|q_{F_{0}}^{F}\circ T\circ \iota_{\langle X_{0}\rangle}^{E}\|_{\I\cap \L}
\end{align*}
We infer that
\begin{align*}
\sup\{\|q_{F_{0}}^{F}\circ T\circ \iota_{X_{0}}^{E}\|_{\I}:X_{0}\in \text{MFIN}(E), F_{0}\in\text{COFIN}(F)\}\leq \|T\|_{(\I\cap \L)^{\max}}
\end{align*}
and therefore $T\in \I^{\max}(E,F)=\I(E,F)$ with $\|T\|_{\I\cap \L}\leq \|T\|_{(\I\cap \L)^{\max}}$.
\end{proof}

As a consequence, we have

\begin{prop}
\label{prop: Lip I reg cap L is equal to Lip I cap L max, where I is ultrastable}
Let $\I$ be an ultrastable Banach Lipschitz operator ideal, then

$$((\Lip_0\circ \I)\cap \L)^{\max}=(\Lip_0\circ \I)^{\reg}\cap \L$$
\end{prop}
\begin{proof}
First, note that if $\J$ is any ultrastable Banach Lipschitz operator ideal, then $\J\cap\L$ is an ultrastable Banach operator ideal. Combining this observation with Proposition~\ref{prop:ultrastable a izquierda}, we obtain that $(\Lip_0\circ \I)\cap \L$ is an ultrastable Banach operator ideal. Then we have 
$$((\Lip_0\circ \I)\cap \L)^{\max}=((\Lip_0\circ \I)\cap \L)^{\reg}=(\Lip_0\circ \I)^{\reg}\cap \L,$$
and the proof is complete.
\end{proof}

The principal tool that we will use to see what Banach operator ideal does $\Lip_0 \circ \A \circ \Lip_0$ extend is the Maurey--Johnson--Schechtman result (Theorem~\ref{thm: JMS lemma about differentiability}). In general, given a linear operator $T \in \Lip_0 \circ \A \circ \Lip_0(E,F)$, we desire a suitable factorization of $T$ through Gateaux differentiable Lipschitz maps. In many cases, that suitable factorization can be obtained using translations. So it is convenient to introduce the following definition. In what follows, for $f\in \Lip_0(E,F)$ and $x_0\in E$, we denote by $f^{x_0}(x):=f(x+x_0)-f(x_0)$.
\begin{definition}
A Banach Lipschitz operator ideal $\I$ is said to be invariant under translations if for any Banach spaces $E$ and $F$, $x_0\in E$ and $f\in \I(E,F)$, the function $f^{x_0}$ belongs to $\I(E,F)$ an $\|f^{x_0}\|_{\I}\leq \|f\|_{\I}$.
\end{definition}
\begin{example} Let $\A$ be a Banach operator ideal, then $\A\circ \Lip_0$ is invariant under translations.
\end{example}
\begin{proof}
Take Banach spaces $E$ and $F$, $x_0\in E$ and $f\in \A\circ \Lip_0(E,F)$. Consider a factorization $f=T\circ g$ where $g\in \Lip_0(E,G)$ and $T\in \A(G,F)$ for some Banach space $G$. Since $T$ is linear, we have that $f^{x_0}=T\circ g^{x_0}$, implying that $f^{x_0}\in \A\circ \Lip_0(E,F)$.
\end{proof}

\begin{prop}
Let $\I$ be a Banach Lipschitz operator ideal invariant under translations. Then  $\Lip_0\circ \I$ is invariant under translations.
\end{prop}
\begin{proof}
Fix Banach spaces $E$ and $F$ and let $f\in \Lip_0\circ \I(E,F)$. Take a factorization $f=g_1\circ g_2$ where $g_1 \in \I(E,G)$ and $g_2 \in \Lip_0(G,F)$, for some Banach space $G$ and consider $x_0 \in E$. Routine arguments show that $f^{x_0}=g_1^{g_{2}(x_0)}\circ g_2^{x_0}$ and since $g_1^{g_2(x_0)}\in \I(E,G)$ and $g_2^{x_0} \in \Lip_0(G,F)$, we obtain the result.
\end{proof}

Combining the above proposition with the example, we have

\begin{corollary} Let $\A$ a Banach operator ideal, then $\Lip_0\circ \A \circ \Lip_0$ is invariant under traslations.
\end{corollary}

\begin{lemma}
\label{lemma: if f belongs to maximal hull, then its differential at one point also belongs}
Let $\I$ be a maximal Banach Lipschitz operator ideal which is invariant under translations. Let $f\in \I(E,F)$ and suppose that $f$ is Gateaux differentiable at some $x_0$ in $E$. Then $Df(x_0) \in \I\cap \L(E,F)$ and $\|Df(x_{0})\|_{\I}\leq \|f\|_{\I}$.
\end{lemma}
\begin{proof}
Take $f\in \I(E,F)$ Gateaux differentiable at $x_0$. Since $D(J_{F}\circ f)(x_{0})=J_{F}\circ Df(x_{0})$ and $\I$ is regular, we can assume that $F$ is a dual space. First, suppose that $f$ is Gateaux differentiable in $x_0=0$. For $t>0$, consider $\phi_{t}\in \Lip_0(E,F)$ defined by $\phi_t(v)=\frac{1}{t}(f(tv))$. Note that if $h_{t}^{E}\colon E\rightarrow E$ is the linear operator defined as $h_{t}^{E}(x)=t x $, the map $\phi_{t}$ can be written as $\phi_{t}=h_{\frac{1}{t}}^{F}\circ  f\circ h_{t}^{E}$.  Since $f$ belongs to $\I$, by the ideal property we infer that $\phi_{t}\in \I(E,F)$ for all $t>0$, and $\|\phi_{t}\|_{\I}\leq \|f\|_{\I}$.

By assumption, $\phi_{t}$ converges pointwise to $Df(0)$ as $t$ tends to $0$. In particular, $\phi_t$ converges in the weak-star Lipschitz operator topology of $\Lip_{0}(E,F)$. Since $\I$ is maximal, using the Representation Theorem of Maximal Banach Lipschitz operator ideal, we can apply \cite[Theorem 6.5]{CPCDJVVV} to obtain that $Df(0)\in \I(E,F)$ and $\|Df(0)\|_{\I}\leq \|f\|_{\I}$.

Now, we treat the case $x_{0}\neq 0$. For this, we consider the map $f^{x_{0}}$, which, by hypothesis, belongs to $\I(E,F)$ and $\|f^{x_{0}}\|_{\I}\leq \|f\|_{\I}$. It is clear that $f^{x_{0}}$ is Gateaux differentiable at $0$ and $Df^{x_{0}}(0)=Df(x_{0})$. By the previous case, we infer that $Df(x_{0})\in \I\cap \L(E,F)$, with 
\begin{align*}
\|Df(x_{0})\|_{\I}=\|Df^{x_{0}}(0)\|_{\I}\leq \|f^{x_{0}}\|_{\I}\leq \|f\|_{\I}
\end{align*}
This completes the proof.
\end{proof}

\begin{prop} Let $\I$ be a maximal Banach Lipschitz operator ideal which is invariant under translations, and let $E$ and $F$ be Banach spaces. Fix $T\in (\Lip_0\circ \I)\cap\L(E,F)$ and suppose that for every $\varepsilon >0$ there exist a Banach space $G$, a function $g\in \I(E,G)$ Gateaux differentiable at some $x_0$ in $E$ and $f\in \Lip_0(G,F)$ such that $T=f\circ g$ and  $\|T\|_{\Lip_{0}\circ \I}\leq \Lip(f)\|g\|_{\I}+\varepsilon$. Then $T\in \I\cap \L(E,F)$ and $\|T\|_{\I}=\|T\|_{\Lip_0\circ \I}$.
\end{prop}
\begin{proof}
For $\varepsilon>0$, pick a factorization $T=f \circ g$ as in the statement and assume that $\Lip(f)=1$. Applying Theorem \ref{thm: JMS lemma about differentiability}, we get an operator $R\in \L(G,F'')$ with $\|R\|\leq 1$ such that $J_F\circ T=R\circ Dg(x_{0})$. By Lemma \ref{lemma: if f belongs to maximal hull, then its differential at one point also belongs}, we have that $Dg(x_{0})\in \I\cap \L(E,G)$, with $\|Dg(x_{0})\|_{\I\cap \L}\leq \|g\|_{\I}$. Therefore $T\in (\I\cap \L)^{\reg}(E,F)=(\I\cap \L)(E,F)$ where the equality holds because $\I\cap \L$ is maximal by Lemma~\ref{lemma: linear restriction of a maximal lipschitz ideal is maximal}. Moreover, $\|T\|_{\I\cap \L}\leq \|g\|_{\I} +\varepsilon$. It follows that $\|T\|_{\I}\leq \|T\|_{\Lip_0 \circ \I}$. Since $\I\subset \Lip_{0}\circ \I$, we also have the reverse inequality.
\end{proof}

\begin{corollary}
Let $\I$ be a maximal Banach Lipschitz operator ideal which is invariant under translations.  Suppose that for every finite-dimensional Banach space $M$ and every Banach space $F$, every Lipschitz function $f\in \I(M,F)$ has a point of Gateaux differentiability in $M$. Then
$$(\Lip_{0}\circ \I)^{\max}\cap \L=\I\cap \L$$
\end{corollary}
\begin{proof}
By Lemma~\ref{lemma: linear restriction of a maximal lipschitz ideal is maximal}, $\I\cap \L$ and $(\Lip_{0}\circ \I)^{\max}\cap \L$ are maximal Banach operator ideals. Therefore, we only have to check that their respective norms coincide over operators between finite-dimensional normed spaces. Fix finite--dimensional Banach spaces $M$ and $N$, and take $T\in \L(M,N)$. Since $\I$ is maximal, by Proposition~\ref{prop:ultrastable a izquierda} and Theorem~\ref{Theo:_Reg_and_Ultrast},  $\Lip_0\circ \I$ is ultrastable. Also, since $N$ is a finite--dimensional space, then $\Lip_0\circ \I(M,N)=(\Lip_0\circ \I)^{\reg}(M,N)$. This all together gives that $\Lip_0\circ \I(M,N)=(\Lip_0\circ \I)^{\max}(M,N)$

 Then, by Proposition~\ref{prop: Lip I reg cap L is equal to Lip I cap L max, where I is ultrastable} and the above proposition, we have that
$$ \|T\|_{(\Lip_{0}\circ \I)^{\max}\cap \L}=\|T\|_{(\Lip_{0}\circ \I)\cap \L}= \|T\|_{\I\cap \L},$$and the proof is complete.
\end{proof}
As an application, we have

\begin{corollary}\label{Coro:Gdif}
Let $\A$ be a maximal Banach operator ideal such the Lipschitz ideal $\A\circ Lip_{0}$ has the property that for every finite-dimensional Banach space $M$ and every Banach space $F$, every Lipschitz function $f\in \A\circ \Lip_0(M,F)$ has a point of Gateaux differentiability in $M$. Then
$$
(\Lip_0\circ \A\circ \Lip_0)\cap \L=(\Lip_0\circ \A\circ \Lip_0)^{\max}\cap \L= \A.
$$
\end{corollary}

\begin{proof}
By \cite[Proposition~3.3]{TV} and \cite[Proposition~4.1]{TV}, $\A\circ \Lip_0$ is a maximal Lipschitz operator ideal and $(\A\circ \Lip_0)\cap \L=\A$. The inclusions
$$\A\circ \Lip_0\subset \Lip_0\circ \A\circ \Lip_0\subset (\Lip_0\circ \A\circ \Lip_0)^{\max}$$
and above corollary complete the proof.
\end{proof}

We can rewrite part of the above corollary in terms of factorizations of linear operators.

\begin{theorem}\label{thm: Factorization}
Let $E$ and $F$ Banach spaces, $\A$ a Banach operator ideal and $T\in \L(E,F)$. Suppose that there is a factorization of $T$ as 
$$
\xymatrix{
E\ar[r]^T \ar[d]_{f_1}& F\\
G_1 \ar[r]^R & G_2 \ar[u]_{f_2}
}
$$
where $G_1$ and $G_2$ are Banach spaces, $f_1$ and $f_2$ are Lipschitz maps, and $R\in \A(G_1,G_2)$. If $\A$ is a maximal Banach operator ideal such the Lipschitz ideal $\A\circ Lip_{0}$ has the property that for every finite-dimensional Banach space $M$ and every Banach space $F$, every Lipschitz function $f\in \A\circ \Lip_0(M,F)$ has a point of Gateaux differentiability in $M$, then $T \in \A(E,F)$.
\end{theorem}

\begin{remark}\label{remark: Ideales} The ideals of $p$-integral operators, $p$-summing operators with $1\leq p<\infty$, the $(p,q)$-dominated operators, with $1<p,q<\infty$, $1/p+1/q\leq 1$ and the $p$-factorable operators with $1<p<\infty$ are maximal and satisfy the hypothesis in Theorem~\ref{thm: Factorization}. Indeed, take a linear operator $T\colon E\rightarrow F$ such that it can be factorized as $T=R\circ S$ where $S\colon E\rightarrow G$ and $R\colon G\rightarrow F$ are linear operators and $G$ is a Banach space with the Radon-Nikod\'ym  property. Then if $f\in \Lip_0(M,E)$ with $M$ a finite-dimensional space, then $T\circ f=R\circ S\circ f$ and, since by \cite[Theorem~6.42]{BL} $S\circ f$ has a point in which is Gateaux differentiable, then $T\circ f$ is Gateaux differentiable at the same point. Now, since 
 for $1<p<\infty$, every $p$-integral, $p$-summing, $(p,q)$-dominated  ($1/p+1/q\leq 1$) or $p$-factorable operator can be factorized via a closed subspace of $L_p(\mu)$ (and therefore with the Radon-Nikod\'ym  property) with two linear operators, we obtain part of the claim. For $p=1$, just use that every $1$-integral or $1$-summing operator is $2$-summing, and in particular, it factorizes through a reflexive Banach space (and therefore, with the Radon-Nikod\'ym property).
\end{remark}

We finish with a brief discussion about in which cases $\Lip_0\circ \A^{\max} \circ \Lip_0$ is of composition type or not. For this, recall the definition of a right-accessible Banach operator ideal. Following \cite[25.2]{DF}, a Banach operator ideal $\A$ is right-accessible if for every finite-dimensional space $M$ and every Banach space $F$, we have $\A(M,F)=\A^{\min}(M,F)$. In particular, every operator ideal named in Remark~\ref{remark: Ideales} is right-accessible.
\begin{lemma} Let $\A$ be a right-accessible Banach operator ideal. Then for every finite pointed metric space $X_0$ and every Banach space $E$, $\Lip_0\circ \A \circ \Lip_0(X_0,E)=\Lip_0\circ \A^{\min} \circ \Lip_0(X_0,E)$. In particular, both ideals have the same minimal kernel and the same maximal hull.
\end{lemma}
\begin{proof}
Let $X_0$ be a finite pointed metric space and $E$ a  Banach space, and $f\in \Lip_0\circ \A \circ \Lip_0(X_0,E)$. Pick a factorization
$$
\xymatrix{
X_0 \ar[r]^{f}\ar[d]_{g_1}& E\\
G_1\ar[r]^{T} & G_2 \ar[u]_{g_2}
}$$
where $G_1$ and $G_2$ are Banach spaces, $g_1$ and $g_2$ are base-point preserving Lipschitz maps, and $T\in \A(G_1,G_2)$. We have to see that $\|f\|_{\Lip_{0}\circ \A^{\min}\circ \Lip_{0}}\leq \Lip(g_2) \|T\|_{\A} \Lip(g_1) $ Since $X$ is finite, we can assume that $F$ is finite-dimensional (for instance, we may replace $G_1$ by the linear span of $g_1(X_0)$). Since $\A$ is right-accessible, we know that $\A^{\min}(G_1,G_2)=\A(G_1,G_2)$ isometrically. Therefore,
\begin{align*}
\|f\|_{\Lip_{0}\circ \A^{\min}\circ \Lip_{0}}\leq \Lip(g_2) \|T\|_{\A^{\min}} \Lip(g_1)=\Lip(g_2) \|T\|_{\A}\Lip(g_1).
\end{align*}
Since the other inclusion always holds, the proof is complete.
\end{proof}

\begin{prop}
Let $\A$ be a right-accessible Banach operator ideal. If $\Lip_0\circ \A \circ \Lip_0$ is of composition type, then $id_{\Ae(E)}\in \A^{\max}(\Ae(E),\Ae(E))$ for every finite-dimensional normed space $E$.
\end{prop}
\begin{proof}
If $\Lip_0\circ \A \circ \Lip_0$ is of composition type, then $(\Lip_0\circ \A \circ \Lip_0)^{\min}$ is also of composition type by \cite[Theorem~4.8]{TV}. By the above, $(\Lip_0\circ \A^{\min} \circ \Lip_0)^{\min}$ is of composition type. The conclusion follows by applying Proposition~\ref{Prop: Minimal of composition type}.
\end{proof}

\subsection*{Acknowledgments} We are indebted to the anonymous referee for her/his selfless suggestions that significantly improved this article.
This project was supported in part by CONICET PIP 1609. N. Albarrac\'{\i}n is also supported by a CONICET doctoral fellowship. P. Turco is also supported by ANCyT-PICT 2019-00080. 


\begin{thebibliography}{99}



\bibitem{ARSPY} D. Achour, P. Rueda, E. A. Sánchez-Pérez, R. Yahi, \textit{Lipschitz operator ideals and the approximation property}, J. Math. Anal. Appl. 436 (2016) 217-236.

\bibitem{ADT} D. Achour, E. Dahia, P. Turco. {\it Lipschitz p-compact mappings}. Monatshefte für Mathematik (2019) 189:595–609.

\bibitem{ARY} D. Achour, P. Rueda, R. Yahi. {\it ($p,\sigma$)-absolutely Lipschitz operators}. Ann. Funct. Anal. 8 (2017), 38-50.

\bibitem{AE56} R.F. Arens and J.~Eells, Jr. \textit{On embedding uniform and topological spaces}, Pacific J. Math. 6 (1956), 397--403. 

\bibitem{BeChe} A. Belacel, D. Chen. {\it Lipschitz $(p,r,s)$-integral operators and Lipschitz $(p,r,s)$-nuclear operators}. J. Math. Anal. Appl. 461 (2018), 1115-1137.  

\bibitem{BL} Y. Benyamini, J. Lindenstrauss. Geometric Nonlinear Functional Analysis, vol. 1,American Mathematical Society Colloquium Publications 48 (American Mathematical Society, Providence, RI, 2000).



\bibitem{CPCDJVVV2}  M. G. Cabrera-Padilla, J. A. Ch\'avez-Dom\'{\i}nguez, A. Jim\'enez-Vargas and M. Villegas-Vallecillos, \textit{Lipschitz tensor product}, Khayyam J. Math. 1 (2015),  185--218.
\bibitem{CPJVVV} M. G. Cabrera-Padilla,  A. Jim\'enez-Vargas, Mois\'es Villegas-Vallecillos, \textit{Maximal Banach ideals of Lipschitz maps}, Ann. Funct. Anal. 7 (2016), 593--608.

\bibitem{CPCDJVVV} M. G. Cabrera-Padilla, J. A. Ch\'avez-Dom\'{\i}nguez, A. Jim\'enez-Vargas and M. Villegas-Vallecillos. \textit{Duality for ideals of Lipschitz maps}, Banach J. Math. Anal. 11 (2017), 108--129.

\bibitem{CPJV} M. G. Cabrera-Padilla, A. Jiménez-Vargas. {\it Lipschitz Grothendieck integral operators}. Banach J. Math. Anal. 9 (2015), no. 4, 34-57.

\bibitem{ChaDo12} J. A. Ch\'avez-Dom\'{\i}nguez. {\it Lipschitz ($q,p$)-mixing operators}. Proc. Amer. Math. Soc. Vol. 140 (2012), 3101--3115.

\bibitem{ChaDo} J. A. Ch\'avez-Dom\'{\i}nguez. {\it Lipschitz factorization through subsets of Hilbert spaces}. J. Math. Anal. Appl. 418 (2014), 344-356.


\bibitem{CZ} D. Chen and B. Zheng. \textit{Lipschitz $p$-integral operators and Lipschitz $p$-nuclear operators}, Nonlinear Anal. 75 (2012), 5270--5282.

\bibitem{CDW} M. Cuth, M. Doucha, P. Wojtaszczyk. \textit{On the structure of Lipschitz free spaces}, Proc. Amer. Math. Soc. Vol 144 (2016),3833--3846.


\bibitem{DaCKr} D. Dacunha-Castelle and J. L. Krivine. {\it Applications des ultraproduits à l'\'etude des espaces et des algèbres de Banach.} Studia Math. 41 (1972), 315--334.

\bibitem{DF} A. Defant and K. Floret. Tensor norms and operator ideals, North-Holland Math. Studies, 176, 1993.


\bibitem{FJ} J. D. Farmer and W. B. Johnson. \textit{Lipschitz $p$-summing operators}, Proc. Amer. Math. Soc. 137 (2009), 2989--2995.

\bibitem{FloHun} K. Floret and S. Hunfeld, \textit{Ultrastability of ideals of homogeneous polynomials and multilinear mappings on Banach spaces}, Proc. Amer. Math Soc. 130 (2001), 1425--1435.

\bibitem{GK} G. Godefroy and N. Kalton, \textit{Lipschitz free Banach spaces}, Studia Math. 159 (2003), 121--141.

\bibitem{Hei} S. Heinrich, {\it Ultraproducts in Banach space theory.} J. Reine Angew. Math. 313 (1980), 72--104.


\bibitem{JVSMVV14} A. Jim\'enez Vargas, J. M. Sepulcre, M. Villegas Vallecillos. {\it Lipschitz compact operators}, J. Math. Anal. Appl. 415 (2014), 889–901.

\bibitem{Joh} W. Johnson. {\it Factoring compact operators}, Israel J. Math. 9 (1971), 337--345
.
\bibitem{JMS} W. Johnson, B. Maurey and G. Schechtman. {\it Non-linear factorization of linear operators.} Bull. Lon. Math. Soc. 41 (2009) 663--668.

\bibitem{Kur} K. K{\"u}rsten, {\it On some questions of A. Pietsch}  II Teor. Funkcii, Funckcional Anal. i. Prilozenia (Kharkov), 29 (1978), 61--73.

\bibitem{LNO} A. Lima, O. Nygaard, E. Oja. {\it Isometric factorization of weakly compact operators and the approximation property.} Israel J. Math. 119 (2000), 325--348.

\bibitem{Pie} A. Pietsch. Operator ideals, VEB Deutscher Verlag der Wissenschaften, Berlin, 1978.

\bibitem{Pie2} A. Pietsch {\it Ultraprodukte von Operatoren in Banachr{\"a}umen}. Math. Nachr. 61 (1974), 123--132.

\bibitem{Roe} J. Roe. Lectures on Coarse Geometry. Amer. Math. Soc. 2003.


\bibitem{KhaSa} K. Saadi. {\it Some properties of Lipschitz strongly $p$-summing operators.} J. Math. Anal. Appl. 423 (2015), 1410-1426.

\bibitem{TV} P. Turco and R. Villafa\~ne. {\it Galois connection between Lipschitz and linear operator ideals and minimal Lipschitz operator ideals.} J. Funct. Anal. 277 (2019), 434--451.


\bibitem{Wea} N. Weaver. Lipschitz Algebras, 2nd edition, World Scientific, 2018.
\end{thebibliography}
\end{document}